\documentclass[reqno]{siamart190516}
\usepackage{hyperref}
\usepackage{amsfonts,graphicx,mathrsfs}
\usepackage{amsmath,amssymb,amsfonts,subfigure,subeqnarray,latexsym,fancybox,comment}
\usepackage{epsfig,color}
\usepackage{multirow}
\numberwithin{equation}{section}

\newtheorem{scheme}{Scheme}
\newtheorem{example}[theorem]{Example}

\newcommand{\thmref}[1]{Theorem~\ref{#1}}
\newcommand{\corref}[1]{Corollary~\ref{#1}}
\newcommand{\secref}[1]{Section~\ref{#1}}
\newcommand{\lemref}[1]{Lemma~\ref{#1}}
\newcommand{\lemrefs}[2]{Lemmas~\ref{#1} and ~\ref{#2}}
\newcommand{\propref}[1]{Proposition~\ref{#1}}
\newcommand{\proprefs}[2]{Propositions~\ref{#1} and ~\ref{#2}}
\newcommand{\rmkref}[1]{Remark~\ref{#1}}

\newcommand{\figref}[1]{Figure~\ref{#1}}

\newcommand{\tabref}[1]{Table~\ref{#1}}

\newcommand{\eq}[1]{\begin{eqnarray}\label{#1}}
\newcommand{\qe}{\end{eqnarray}}

\newcommand{\nn}{\nonumber}

\newcommand{\be}{\begin{eqnarray}}
\newcommand{\ee}{\end{eqnarray}}
\newcommand{\bal}{\begin{aligned}}
\newcommand{\eal}{\end{aligned}}
\newcommand{\bes}{\begin{eqnarray*}}
\newcommand{\ees}{\end{eqnarray*}}
\newcommand{\bs}{\begin{subeqnarray}}
\newcommand{\es}{\end{subeqnarray}}
\newcommand{\bss}{\begin{subeqnarray*}}
\newcommand{\ess}{\end{subeqnarray*}}

\newcounter{saveeqn}

\def\Span{\operatorname{Span}}

\def\mcF{\mathcal F}
\def\mcK{\mathcal K}
\def\mcL{\mathcal L}
\def\mcM{\mathcal M}
\def\mcS{\mathcal S}
\def\mcA{\mathcal A}
\def\mcX{\mathcal X}
\def\mcY{\mathcal Y}
\def\mcZ{\mathcal Z}

\def\BbbR{\mathbb R}

\def\tt{T}

\def\p{\partial}

\def\sig{\sigma}

\def\O{\Omega}

\def\NChz{{\mathcal NC}^h_0}
\def\NChz2d{{[\mathcal{NC}}^h_0]^2}
\def\tNChz2d{\widetilde {[\mathcal{NC}}^h_0]^2}

\def\eps{\mathfrak{e}}

\def\Im{\operatorname{Im}}
\def\and{\,\text{and}\,}
\def\q{\quad}

\def\<{\left\langle}
\def\>{\right\rangle}

\def\bS{\mathbf S}

\def\bu{\mathbf u}

\def\b1{\mathbf 1}

\def\mfA{\mathfrak A}
\def\mfB{\mathfrak B}
\def\mfE{\mathfrak E}
\def\Tau{{\mathcal T}}

\def\dim{\operatorname{dim}\,}

\def\ds{\, \operatorname{ds}}

\def\forany{\,\forall}
\def\Ind{\operatorname{Ind}\,}

\newcommand{\secrefs}[2]{Sections~\ref{#1} and~\ref{#2}}

\newcommand{\vertiii}[1]{{\left\vert\kern-0.25ex\left\vert\kern-0.25ex\left\vert #1 \right\vert\kern-0.25ex\right\vert\kern-0.25ex\right\vert}}

\numberwithin{figure}{section}
\numberwithin{table}{section}

\newsiamthm{claim}{Claim}
\newsiamremark{remark}{Remark}
\newsiamremark{hypothesis}{Hypothesis}
\crefname{hypothesis}{Hypothesis}{Hypotheses}
\begin{document}
\markboth{Periodic $P_1$--nonconforming quadrilateral element}{Yim \& Sheen}
\title{$P_1$--Nonconforming Quadrilateral Finite Element Space with
  Periodic Boundary Conditions:\\
Part I. Fundamental results on dimensions, bases, solvers, and error analysis}
\author{Jaeryun Yim\thanks{Interdisciplinary Program in
Computational Science \& Technology, Seoul National University,
Seoul 08826, Korea;
{\it current address:} Encored Technologies,
215 Bongeunsa-ro, Gangnam-gu, Seoul 06109, Korea.
\email{jaeryun.yim@gmail.com}}
\and
Dongwoo Sheen
\thanks{Department of Mathematics, Seoul National University, Seoul 08826, Korea.
\email{sheen@snu.ac.kr}}}

\maketitle

\allowdisplaybreaks

\newif\iflong
\longfalse

\begin{abstract}
The $P_1$--nonconforming quadrilateral finite element space with periodic
boundary condition is investigated. The dimension and basis for the
space are characterized
with the concept of minimally essential discrete boundary conditions.
We show that the situation is totally different based on the parity of
the number of discretization on coordinates. Based on the analysis on the space,
we propose several numerical schemes for elliptic problems with
periodic boundary condition. 
Some of these numerical schemes are related with solving a linear equation
consisting of a non-invertible matrix. By courtesy of the Drazin inverse,
the existence of corresponding numerical solutions is guaranteed.
The theoretical relation between the numerical solutions is derived,
and it is confirmed by numerical results.
Finally, the extension to the three dimensional is provided.
\end{abstract}


\section{Introduction}
Many macroscopic material properties are obtained from the knowledge
of accurate microscopic material properties. However in most realistic cases
the ratio of macro scale to micro scale is so large that it cannot be
directly computed the dynamics described at the microscale
level. Therefore usually upscaling techniques are used to reduce the micro scale
level computation to approximately obtain macroscopic properties.
Recently, several efficient multiscale methods have been developed
towards that direction. These include numerical homogenization \cite{babuska1976homogenization, babuska1994special, efendiev2004numerical, engquist2008asymptotic}, 
MsFEM (multiscale finite element methods) and GMsFEM (generalized MsFEM)
\cite{efendiev2009multiscale, efendiev2013generalized, efendiev2014generalized, hou1997multiscale}, VMS (variational multiscale finite element methods) \cite{hughes1998variational},
MsFVM (multiscale finite volume methods) \cite{jenny2003multiscale}
and HMM (heterogeneous multiscale 
methods) \cite{abdulle2012heterogeneous, e2003heterogeneous}.
In numerical homogenization and upscaling of multiscale problems one often
needs to solve periodic boundary value problems at microscale level efficiently.

The $P_1$--nonconforming quadrilateral finite element
\cite{park-sheen-p1nc} has an advantage in computing stiffness matrice
as the gradient of linear polynomials is constant in each quadrilateral
as well as it has the smallest number of DOFs (degrees of freedom) for given
quadrilateral mesh.
There have been a number of studies about this finite element 
for fluid dynamics, elasticity, electromagnetics \cite{kim-yim-sheen,
  feng-kim-nam-sheen-stabilized-stokes, nam-choi-park-sheen-cheapest,
  lim-sheen-driven-cavity, park-locking,
  shi-pei-low-nonconforming-maxwell,
  carstensen-hu-unifying-posteriori, feng-li-he-liu-p1nc-fvm}.
Unlikely other finite elements,
this space is strongly tied with the boundary condition for given problem
due to the dice rule constraint element by element (See \eqref{eq:dice-rule}).
Most of those works are focused on the finite element space with
Dirichlet and/or Neumann BCs.
Altmann and Carstensen \cite{altmann-carstensen-p1nc-tri-quad} show
the dimension of, and a basis for 
the finite element space with inhomogeneous Dirichlet BCs
which shares similar discrete nature with the Neumann boundary case.

On the other hand, the $P_1$-nonconforming quadrilateral element space
with periodic BC has not been investigated.
Thus, it is our intention to investigate its dimension and
basis with periodic BC.

The discrete formulation of periodic problems yields singular linear systems,
which can be dealt with various kinds of generalized inverses of a matrix.
Among them, we will concentrate on the Drazin inverse,
as it can be expressed as a matrix polynomial,
since the Krylov method, which is based on the same idea on matrix
polynomials, can be applied to singular linear systems.
%
One of the most important properties of the Drazin inverse
is the expressibility it as a polynomial in the given matrix.
The Krylov iterative method for a nonsingular linear system is established on this property.
The Krylov scheme can be applied to a singular linear system as well
under proper consistency conditions \cite{ipsen-meyer, kaasschieter-pcg, zhang-lu-wei-uzawa-singular, campbell-meyer, axelsson-iterative, bochev-lehoucq-pure-neumann}.

The aim of this paper is to investigate the structure of the
$P_1$--nonconforming quadrilateral finite element spaces with periodic
BC thoroughly and to suggest some iterative methods to
solve the resulting linear systems based
on the idea of Drazin inverse. An application for nonconforming
heterogeneous multiscale methods (NcHMM) of $P_1$--nonconforming
quadrilateral finite element will appear in \cite{yim-sheen-sim-nchmm}.

The organization of the paper is as follows. In \secref{sec:prelim},
we give a brief explanation for the $P_1$--nonconforming quadrilateral
finite element and the Drazin inverse.
We investigate the dimension of the finite element spaces with various BCs,
including periodic condition which is our main concern, in
\secref{sec:dim-finite-space}. 
We introduce the concept of {\it minimally essential discrete
    BCs} to analyze the precise effects of given BC on the
dimension of the corresponding finite element space.
In \secref{sec:basis-finite-space-periodic}, 
a basis for the periodic nonconforming finite element space is constructed.
It consists of node-based functions by identifying
boundary node-based functions in a suitable way 
and a complementary basis consisting of a few alternating functions is considered.
We propose several numerical schemes for solving
a second-order elliptic problem with periodic BC in
\secref{sec:numerical-scheme}. 
We use an efficient iterative method based on the Krylov space 
in help of the Drazin inverse of the corresponding singular matrix.
The relationship between solutions of the schemes will be discussed.
Finally, we extend all our results to the 3D case in \secref{sec:3d}.

\section{Preliminaries and notations}\label{sec:prelim}
In this section some basics on the Drazin inverse and
the $P_1$--nonconforming quadrilateral finite element
will be briefly reviewed. Also notations to be used are described.

\subsection{The Drazin inverse}\label{sec:drazin-inverse}
Let $A$ be a linear transformation on $\mathbb{C}^n$.
The {\it index of $A$,} denoted by $\Ind(A),$ is defined as 
the smallest nonnegative integer $k$ such that
$$\Im A^0 \supset \Im A \supset \cdots \supset \Im A^{k-1} \supset \Im A^k = \Im A^{k+1} = \cdots,$$
or equivalently
$$\ker A^0 \subset \ker A \subset \cdots \subset \ker A^{k-1} \subset \ker A^k = \ker A^{k+1} = \cdots.$$
It yields that, restricted on $\Im A^k$,
the transformation $A$ becomes an invertible linear transformation.
The Drazin inverse of $A$, denoted by $A^D$, is defined as follows:
for $u = v + w \in \mathbb{C}^n$ where $v\in \Im A^k$ and $w\in \ker A^k$,
$A^D u := \left.A\right|_{\Im A^k}^{-1} v$.
One of the most important properties of the Drazin inverse matrix of $A$ is that
it is expressible as a polynomial in $A$:
\begin{theorem}[\cite{campbell-meyer}]
If $A \in \mathbb{C}^{n\times n}$,
then there exists a polynomial $p(x)$ such that $A^D = p(A)$.
\end{theorem}
For a singular matrix $A,$
a unique Drazin inverse solution can be found by using the
Krylov iterative method under some proper consistency conditions.
For details, see \cite{campbell-meyer, ipsen-meyer}.

\begin{theorem}[\cite{ipsen-meyer}]
Let $m$ be the degree of the minimal polynomial for $A$,
and let $k$ be the index of $A$.
If $b\in \Im A^k$, then the linear system $Ax=b$ has a unique Krylov solution $x=A^D b \in \mcK_{m-k}(A,b)$.
If $b\not\in \Im A^k$, then $Ax=b$ does not have a solution in the Krylov space $\mcK_n(A,b)$.
\end{theorem}

\subsection{Notations}
For $d=2$ or $3,$ let $\O=\Pi_{j=1}^d (0,\ell_j)\subset \BbbR^d$ denote a
$d$-dimensional rectangular domain.
Let $(\Tau_h)_{0<h<\min_{j=1}^d(\ell_j)}$ be
the quasiuniform family of triangulations of $\O$
into $d$--dimensional polyhedral subdomains $Q_h$'s
which are convex and topologically equivalent to $d$--dimensional cubes,
with maximum diameter bounded by the mesh parameter $h.$
We further assume that, for ${0<h<\min_{j=1}^d(\ell_j)},$
$\Tau_h$ is topologically and combinatorially
equivalent to the $N_{x_1}\times\cdots\times N_{x_d}$ uniform
$d$--dimensional rectangular decomposition, say $\widetilde\Tau_h.$
We will call that the sequences of elements, faces, and vertices in $\Tau_h$ are
aligned in the {\it topological $x_k$-direction} we mean they are images
of elements, faces, and vertices in $\widetilde\Tau_h$ aligned in the $x_k$-direction.

Let $\mcF_h$, $\mcF^i_h$, $\mcF^b_h$, and $\mcF^{b,opp}_h$ denote
the sets of all $(d-1)$--dimensional faces, interior faces, boundary faces,
and boundary face pairs on opposite boundary position, respectively.
Let $\mathcal{N}_h$ denote the set of all nodes in $\Tau_h$.
For periodic BC, we assume that for each $h$, $\Tau_h$ is decomposed
such that the periodically opposite boundary pairs in $\mcF^{b,opp}_h$ 
are congruent.

From now on, for each face $f$, let $\sig^{(\iota)}_f,\iota=i,m,$ denote the
functionals which take the face average value and the midpoint value
at the face midpoint $m_f$, respectively, such that
$\sig^{(i)}_f(u) = \frac1{|f|}\int_f u\ds$ and
$\sig^{(m)}_f(u) = u(m_f)$ for given function $u.$
We adopt several standard Sobolev spaces and discrete function spaces 
for the $P_1$--nonconforming quadrilateral finite element:
\begin{align*}
&C^\infty_{\#}(\O) = \text{the subset of $C^\infty(\BbbR^d)$
                        of $\O$-periodic functions restricted to }\O, \\
&H^1_{\#}(\O) = \overline{C^\infty_{\#}(\O)}^{H^1(\O)}, \q
H^1_{\#}(\O)/\BbbR = \{ v \in H^1_{\#}(\O) ~|~ \int_{\O} v = 0 \}, \\
&V^h = \{ v_h \in L^2(\O) ~|~ \left.v_h\right|_K \in \mathcal{P}_1(K) \forany K\in \Tau_h,
	\, \sig^{(i)}_f([v_h]_f) = 0 \forany f \in \mcF^i_h\}, \\
&V^h_0 = \{ v_h \in V^h ~|~  \sig^{(i)}_f(v_h) = 0 \forany f \in \mcF^b_h\}, \q
V^h_{\#} = \{ v_h \in V^h ~|~  \sig^{(i)}_{f_1}(v_h) =
                                                                    \sig^{(i)}_{f_2}(v_h)
  \\
  &\q\q \forany (f_1,f_2) \in \mcF^{b,opp}_h\}, \q V^h_{\#}/\BbbR = \{ v_h \in V^h_{\#} ~|~ \int_{\O} v_h = 0 \},
\end{align*}
where $\mathcal{P}_1(K)$ denotes the space of all linear polynomials on $K$
and $[\cdot]_f$ the jump across $(d-1)$-dimensional face $f$.
Let $\|\cdot\|_0$, $|\cdot|_1$, and $|\cdot|_{1,h}$ denote
the standard $L^2$-norm, $H^1$-(semi-)norm, and mesh-dependent energy norm in $\O$, respectively.

Here we define the concept of node--based functions.
For a given node $z$ in $\Tau_h$, let $\mcF_{(z)}$ denote the set of
all $(d-1)$-dimensional faces containing $z$.
Then we can construct a function $\phi_z \in V^h$ associated with $z$ such that
$
\sig_f^{(m)}(\phi_z) =
\begin{cases}
\frac12 & \text{if } f \in \mcF_{(z)}, \\
0 & \text{otherwise},
\end{cases}
$
where $m_f$ is the midpoint of $(d-1)$-dimensional face $f$ in $\mcF_h$.
We call $\phi_z$ {\it the node--based function associated with $z$.}
In the case of periodic BC for a rectangular domain $\O$,
we identify two boundary nodes in every opposite periodic position,
and four nodes at the corners of the boundary.
Using the node--based functions, we introduce a discrete function
space and a set of functions, which will be used often:
\begin{subeqnarray}
V^{\mfB,h}_{\#} &=& \{ v_h \in V^h_{\#} ~|~ v_h \in \Span \mfB \}, \\
\mfB &=& \{ \phi_z \}_{z \in \mathcal{N}^{\#}_h}: \text{the set of all node--based functions in $V^h_{\#}$},
\end{subeqnarray}
where $\mathcal{N}^{\#}_h$ denotes the set of all nodes with periodical identification. 
Notice that $|\mfB| = N_x N_y$ in the 2D case,
$|\mfB| =N_x N_y N_z$ in the 3D case, due to identification between nodes on boundary.
For $\mathfrak{S} \subset L^\infty(D),$ of size $|\mathfrak{S}|$ and
a scalar-valued (integrable) function $f,$ 
$\int_\mathcal{D} f \mathfrak{S}$ denotes a vector, of size $|\mathfrak{S}|$,
such that each component is the integral of the product of
$f$ and the corresponding element in $\mathfrak{S}$ over the domain $\mathcal{D}$.
$\mathbf{1}_{\mathfrak{S}}$ denotes a vector, size of $|\mathfrak{S}|$,
consisting of $1$ for all components.

\section{Dimension of the Finite Element Spaces}\label{sec:dim-finite-space}
\subsection{Induced relation between boundary barycenter values}
We consider a finite element space which approximates given function
space with given BC.
Then the barycenter values on boundary faces
in the $P_1$--nonconforming quadrilateral element space satisfy the
following condition: for all $u\in V^h$ 
\begin{eqnarray}\label{eq:dice-rule}
  \sig^{(m)}_{f_1}(u) + \sig^{(m)}_{f_1^{\text{opp}}}(u)  = \cdots =
  \sig^{(m)}_{f_d}(u) + \sig^{(m)}_{f_d^{\text{opp}}}(u)
\end{eqnarray}
for all pairs
$(f_j,f_j^{\text{opp}}) \in \mcF^{b,opp}_h(Q)$ for all $Q\in\Tau_h,$
where
$\mcF^{b,opp}_h(Q)$ denotes the set of all pairs consisting of
two boundary faces on opposite position. We will coin the above
formula \eqref{eq:dice-rule} as {\it the dice rule.}

We will concentrate on the case of $d=2$ in this section,
and Sections \ref{sec:basis-finite-space-periodic}--\ref{sec:numerical-scheme}.
The 3 dimensional case will be covered in \secref{sec:3d}.

Let $N_Q$ denote the number of all elements in $\Tau_h$.
Let $N_V$, $N_V^i$, and $N_V^b$ denote the number of all vertices,
of all interior vertices, and of all boundary vertices, respectively.
Similarly $N_E$, $N_E^i$, and $N_E^b$ denote the number of all edges,
of all interior edges, and of all boundary edges, respectively.
%
The vertices in $\Tau_h$ are grouped into Red and Black groups
such that any two vertices connected by an edge in $\Tau_h$
are not contained in the same group.

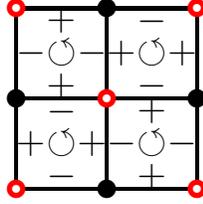
\begin{figure}[!ht]
\centering
\setlength{\unitlength}{12mm}
\begin{picture}(2,2)(0,-0.2)
\linethickness{0.4mm}
\multiput(0,0)(1,0){3}{\line(0,1){2}}
\multiput(0,0)(0,1){3}{\line(1,0){2}}

\Large
\put(0.0, 0.5){\makebox(0.0,0.0)[l]{$+$}}
\put(0.5, 0.0){\makebox(0.0,0.0)[b]{$-$}}
\put(1.0, 0.5){\makebox(0.0,0.0)[r]{$+$}}
\put(0.5, 1.0){\makebox(0.0,0.0)[t]{$-$}}
\put(0.5, 0.5){\makebox(0.0,0.0)[c]{$\circlearrowleft$}}
\put(1.0, 0.5){\makebox(0.0,0.0)[l]{$-$}}
\put(1.5, 0.0){\makebox(0.0,0.0)[b]{$+$}}
\put(2.0, 0.5){\makebox(0.0,0.0)[r]{$-$}}
\put(1.5, 1.0){\makebox(0.0,0.0)[t]{$+$}}
\put(1.5, 0.5){\makebox(0.0,0.0)[c]{$\circlearrowleft$}}
\put(0.0, 1.5){\makebox(0.0,0.0)[l]{$-$}}
\put(0.5, 1.0){\makebox(0.0,0.0)[b]{$+$}}
\put(1.0, 1.5){\makebox(0.0,0.0)[r]{$-$}}
\put(0.5, 2.0){\makebox(0.0,0.0)[t]{$+$}}
\put(0.5, 1.5){\makebox(0.0,0.0)[c]{$\circlearrowleft$}}
\put(1.0, 1.5){\makebox(0.0,0.0)[l]{$+$}}
\put(1.5, 1.0){\makebox(0.0,0.0)[b]{$-$}}
\put(2.0, 1.5){\makebox(0.0,0.0)[r]{$+$}}
\put(1.5, 2.0){\makebox(0.0,0.0)[t]{$-$}}
\put(1.5, 1.5){\makebox(0.0,0.0)[c]{$\circlearrowleft$}}

\put(1.0, 0.0){\circle*{0.2}}
\put(0.0, 1.0){\circle*{0.2}}
\put(2.0, 1.0){\circle*{0.2}}
\put(1.0, 2.0){\circle*{0.2}}
\color{red}
\put(0.0, 0.0){\circle*{0.2}}
\put(2.0, 0.0){\circle*{0.2}}
\put(1.0, 1.0){\circle*{0.2}}
\put(0.0, 2.0){\circle*{0.2}}
\put(2.0, 2.0){\circle*{0.2}}
\color{white}
\put(0.0, 0.0){\circle*{0.1}}
\put(2.0, 0.0){\circle*{0.1}}
\put(1.0, 1.0){\circle*{0.1}}
\put(0.0, 2.0){\circle*{0.1}}
\put(2.0, 2.0){\circle*{0.1}}
\end{picture}
\caption{For each element, the signs on its edges are chosen 
  $+$ if the edges are from a Black to Red vertices, and $-$ otherwise.
}
\label{fig:dice-rule-same-orientation}
\end{figure}
A fixed orientation of edges is chosen throughout the all elements in $\Tau_h.$
For instance, we impose the plus sign on an edge if its direction is
from Red to Black, and the minus sign if the direction is opposite.
The local signs on edges in each element induce a relation between
4 midpoint values on the element which corresponds to the dice rule:
\begin{align*}
  \sum_{j=1}^4 (-1)^j  \sig_{f_j^K}^{(m)}(v_h) = 0,\,  \forany v_h\in V^h ,\,
f_j^K\text{ being the }j^{th} \text{ edge of } K,
\forany K\in\Tau_h,
\end{align*}
Since two local signs on both sides of an interior edge are always opposite,
the sum of all locally induced relations reduces to a relation
between midpoint values on boundary edges only.
Note that the number of boundary edges in $\Tau_h$ is always even
and the remaining signs are alternating along the boundary.
\figref{fig:dice-rule-same-orientation} shows an example of
orientation and induced signs on edges. The following lemma is easy
but essential to the nonconofrming $P_1$ element $V^h.$
\begin{lemma}\label{lem:relation-boundary-dof}
There exists a way to give alternating sign on boundary edges.
Moreover, the alternating sum of boundary midpoint values of
  $v_h \in V^{h}$ is always zero, whenever the domain is simply connected.
\end{lemma}
\subsection{Minimally essential discrete BCs}
Among all the midpoint values of a given essential BC
only a subset of them is enough to impose consistent discrete
boundary values.
We call a set of discrete BCs {\it minimally essential}
if essential boundary midpoint values in the set induce all other
essential boundary midpoint values naturally,
but any proper subset of the set does not.

Since each discrete essential BC removes
the dimension of the space by 1,
the number of subtracted DOFs due to essential BCs
is just equal to the number of minimally essential discrete BCs.
It recovers a well-known fact for the dimension of the finite element spaces with Neumann
and homogeneous Dirichlet BC.

\begin{lemma}\label{lem:dim-formula-2d}
$  \dim(V^h)=
	 N_E - N_Q 
	- \#(\text{minimally essential discrete BCs}).
$
\end{lemma}

\begin{proposition}For Neumann and Dirichlet BCs, we have
\begin{align*}
\# \text{(minimally essential discrete BCs)}
=
\begin{cases}
0 & \text{for Neumann BC,} \\
N^b_E -1 & \text{for homogeneous Dirichlet BC}.
\end{cases}
\end{align*}
Consequently,
$\dim V^h = N_E-N_Q = N_V-1,$ and $\dim V^h_0 = N_E-N_Q-(N_E^b-1) = N_V^i.$
\end{proposition}
\begin{remark} The proposition generalizes the dimensions for the
  homogeneous Dirichlet and Neumann BCs given in
Theorems 2.5 and 2.8 in \cite{park-sheen-p1nc}.
\end{remark}

For periodic BCs,
the conditions enforce two midpoint values
on two opposite boundary edges to be equal.
Therefore minimally essential discrete BCs form a smallest set of
periodic relations between opposite boundary edges
which induce all such periodic relations.

Depending on the parity of $N_x$ and $N_y$,
the behavior varies.
\begin{enumerate}
  \item[Case 1.] First, suppose both $N_x$ and $N_y$ are even.
We can easily derive the last periodic relation
from the other periodic relations
with the help of the relation between boundary midpoint values in \lemref{lem:relation-boundary-dof}.
This means that a set of all periodic relations except any one of them
is minimally essential.
  \item[Case 2.]
Next, consider the case where
either $N_x$ or $N_y$ is odd.
Then we can not have such a natural induction as in the Case 1,
which means that a set of all periodic relations itself is minimally essential,
see \figref{fig:induced-relation-boundary-values-even-odd}.
\end{enumerate}

\begin{figure}[!ht]
\centering
\setlength{\unitlength}{6mm}
\begin{picture}(4,4)(0,-0.5)
\linethickness{0.4mm}
\multiput(0,0)(1,0){5}{\line(0,1){4}}
\multiput(0,0)(0,1){5}{\line(1,0){4}}

\scriptsize
\put(0.0, 0.5){\makebox(0.0,0.0)[r]{$+$}}
\put(0.0, 1.5){\makebox(0.0,0.0)[r]{$-$}}
\put(0.0, 2.5){\makebox(0.0,0.0)[r]{$+$}}
\put(0.0, 3.5){\makebox(0.0,0.0)[r]{$-$}}
\put(0.5, 4.0){\makebox(0.0,0.0)[b]{$+$}}
\put(1.5, 4.0){\makebox(0.0,0.0)[b]{$-$}}
\put(2.5, 4.0){\makebox(0.0,0.0)[b]{$+$}}
\put(3.5, 4.0){\makebox(0.0,0.0)[b]{$-$}}
\put(4.0, 3.5){\makebox(0.0,0.0)[l]{$+$}}
\put(4.0, 2.5){\makebox(0.0,0.0)[l]{$-$}}
\put(4.0, 1.5){\makebox(0.0,0.0)[l]{$+$}}
\put(4.0, 0.5){\makebox(0.0,0.0)[l]{$-$}}
\put(3.5, 0.0){\makebox(0.0,0.0)[t]{$+$}}
\put(2.5, 0.0){\makebox(0.0,0.0)[t]{$-$}}
\put(1.5, 0.0){\makebox(0.0,0.0)[t]{$+$}}
\put(0.5, 0.0){\makebox(0.0,0.0)[t]{$-$}}
\end{picture}
\hspace{3cm}
\begin{picture}(3,4)(0,-1)
\linethickness{0.4mm}
\multiput(0,0)(1,0){4}{\line(0,1){3}}
\multiput(0,0)(0,1){4}{\line(1,0){3}}

\scriptsize
\put(0.0, 0.5){\makebox(0.0,0.0)[r]{$+$}}
\put(0.0, 1.5){\makebox(0.0,0.0)[r]{$-$}}
\put(0.0, 2.5){\makebox(0.0,0.0)[r]{$+$}}
\put(0.5, 3.0){\makebox(0.0,0.0)[b]{$-$}}
\put(1.5, 3.0){\makebox(0.0,0.0)[b]{$+$}}
\put(2.5, 3.0){\makebox(0.0,0.0)[b]{$-$}}
\put(3.0, 2.5){\makebox(0.0,0.0)[l]{$+$}}
\put(3.0, 1.5){\makebox(0.0,0.0)[l]{$-$}}
\put(3.0, 0.5){\makebox(0.0,0.0)[l]{$+$}}
\put(2.5, 0.0){\makebox(0.0,0.0)[t]{$-$}}
\put(1.5, 0.0){\makebox(0.0,0.0)[t]{$+$}}
\put(0.5, 0.0){\makebox(0.0,0.0)[t]{$-$}}
\end{picture}
\caption{Induced relation between boundary midpoint values}
\label{fig:induced-relation-boundary-values-even-odd}
\end{figure}
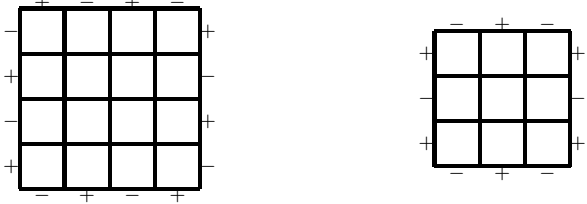

We summarize the above as in the following proposition:
\begin{proposition}(Periodic BC)\label{prop:dim-periodic}
In the case of periodic BC on $N_x \times N_y$ rectangular mesh,
$
 \# \text{(minimally essential discrete BCs)}=N_x + N_y -\eps(N_x)\eps(N_y).
$
Consequently, $
\dim V^h_{\#} = N_x N_y + \eps(N_x)\eps(N_y),$ where
$ \eps(j) := \frac{1+(-1)^{j}}{2}.$
\end{proposition}

%
%

\section{Bases for Finite Element Spaces with Periodic BC}\label{sec:basis-finite-space-periodic}
In this section, we investigate bases for $V^h_{\#}$.
\subsection{Linear dependence of $\mfB$}\label{sec:dependence-B}
We write $\mfB = \{ \phi_{z_1}, \phi_{z_2}, \cdots, \phi_{z_{|\mfB|}}
\},$
the set of all  node--based functions in $V^h_{\#}.$
Define a surjective linear map
$B^{\mfB}_h:\BbbR^{|\mfB|} \rightarrow V^{\mfB,h}_{\#}$
by $B^{\mfB}_h(\mathbf{c}) = \sum_{j=1}^{|\mfB|} c_j \phi_{z_j}$
where $\mathbf{c} = (c_j) \in \BbbR^{|\mfB|}$.
For any $\mathbf{c} = (c_j) \in \ker B^{\mfB}_h$,
we have
\begin{eqnarray}\label{eq:chain-rel}
  c_k = -  c_\ell
\text{ for all vertex pair } (z_k, z_\ell) \text{ which are two end nodes of an edge.}
\end{eqnarray}
\eqref{eq:chain-rel} means that $\dim\ker B^{\mfB}_h\le 1.$
Due to the periodicity,
those relations are consistent only if the number of discretization on each coordinate is even,
and in such a case $\dim\ker B^{\mfB}_h=1.$
Indeed, in this case any $|\mfB|-1$ functions in $\mfB$ form a basis
for $V^{\mfB,h}_{\#}.$
On the other hand, consider the case where either $N_x$ or $N_y$ is odd.
Without loss of generality, we may assume that $N_x$ is odd. Then
a chain of such relation \eqref{eq:chain-rel} 
along the $x$--direction cannot occur unless $\mathbf{c}$ is trivial
since the values at four values at the corners of $\O$ should match.
This concludes that $\dim\ker B^{\mfB}_h = 0.$
We summarize the above result as the following proposition.
\begin{proposition}\label{prop:dim-ker}(The dimension of $\ker B^{\mfB}_h$ and $V^{\mfB,h}_{\#}$)
\begin{align}
\dim \ker B^{\mfB}_h = \eps(N_x)\eps(N_y).
\end{align}
Moreover, $\mfB^\flat= \{\phi_1, \cdots , \phi_{|\mfB| -1}\}$ forms a basis for $V^{\mfB,h}_{\#}$
if both $N_x$ and $N_y$ are even,
whereas $\mfB$ itself is a basis for $V^{\mfB,h}_{\#}$
if either $N_x$ or $N_y$ is odd.
Consequently,
\begin{align}
\dim V^{\mfB,h}_{\#}= |\mfB|-\dim \ker B^{\mfB}_h= N_xN_y -\eps(N_x)\eps(N_y).
\end{align}
\end{proposition}

\subsection{A basis for $V^h_{\#}$}\label{sec:alt-basis}
First, consider the case where both $N_x$ and $N_y$ are even.
\proprefs{prop:dim-periodic}{prop:dim-ker} imply
that $\mfB$ is linearly dependent and 
$V^{\mfB,h}_{\#}$ is a proper subset of $V^h_{\#}$ with
$ \dim(V^h_{\#})- \dim (V^{\mfB,h}_{\#}) =2\eps(N_x)\eps(N_y)=2$,
which means that there exist two {\it complementary basis functions}
for $V^h_{\#}\setminus V^{\mfB,h}_{\#}$.
Let us construct such basis functions.
Define $\psi_x\in V^h_{\#}$ such that
\begin{eqnarray}\label{eq:psi_x}
 & \text{ its midpoint values on {topologically} vertical edges
   are } \pm 1  \nn\\
& \text{ with alternating sign in both directions and } \\
 & \text{ all the midpoint values on {topologically} horizontal edges are 0}.\nn
\end{eqnarray}  
See \figref{fig:alternating-function} $(a)$ for an illustration for
$\psi_x$. Notice that
$\psi_x$ is well-defined whenever $N_x$ is even.
It is easy to see that $\psi_x \not \in V^{\mfB,h}_{\#}.$
Similarly, we can find another piecewise linear function $\psi_y$ in $V^h_{\#},$
not belonging to $V^{\mfB,h}_{\#}$ (\figref{fig:alternating-function}
$(b)$), such that its midpoint values on {topologically} horizontal edges are $\pm 1$
with alternating sign in both directions and all the midpoint values
on {topologically} vertical edges are 0.
\begin{figure}[!ht]
\centering
\epsfig{figure=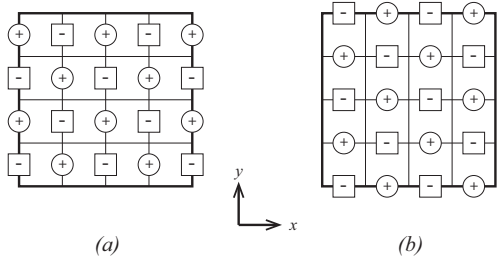,width=0.5\textwidth}
\caption{An example of two alternating functions $(a)$ $\psi_x$ and $(b)$ $\psi_y$. They do not belong to $\mfB$.}
\label{fig:alternating-function}
\end{figure}
Next, let us consider the case where either $N_x$ or $N_y$ is odd.
\proprefs{prop:dim-periodic}{prop:dim-ker} imply that
$\mfB$ is linearly independent and 
$\dim V^{\mfB,h}_{\#} = \dim V^h_{\#}.$
Therefore $V^{\mfB,h}_{\#} = V^h_{\#}$
and $\mfB$, the set of all node--based functions, is a basis for $V^h_{\#}$.
We summarize these results as in following theorem.
\begin{theorem}\label{thm:B}(A basis for $V^h_{\#}$)
\begin{enumerate}
\item If both $N_x$ and $N_y$ are even,
then $V^{\mfB,h}_{\#} \subsetneqq V^h_{\#}$.
Furthermore $\mfA=\{\psi_x, \psi_y\}$ where $\psi_x$
and $\psi_y$ are defined as in \eqref{eq:psi_x}, forms a
complementary basis for $V^h_{\#},$ 
not belonging to $V^{\mfB,h}_{\#}$.
Moreover,
$\mfB^\flat \cup \mfA$ forms a basis for $V^h_{\#},$ where
$\mfB^\flat= \{ \phi_1, \cdots , \phi_{|\mfB| -1}\}$.
\item If either $N_x$ or $N_y$ is odd,
then $V^{\mfB,h}_{\#} = V^h_{\#}$.
Moreover, $\mfB$ is a basis for $V^h_{\#}$.
\end{enumerate}
\end{theorem}

\begin{remark}
  Notice that the elementwise derivatives $\frac{\p\psi_x}{\p x}$
  and $\frac{\p\psi_y}{\p y}$ are checkerboard patterns, while
   $\frac{\p\psi_x}{\p y}= \frac{\p\psi_y}{\p x}=0.$
\end{remark}

\subsection{Stiffness matrix associated with $\mfB$}\label{sec:stiffness-B}
Even though $\mfB$ may not be a basis for $V^h_{\#}$,
it is still a useful set of functions to understand $V^h_{\#}$.
Above all, the node--based functions are easy to handle in implementation viewpoint.
Furthermore, \thmref{thm:B} implies
$V^{\mfB,h}_{\#},$ which equals to $\Span(\mfB)$, occupies almost all of $V^h_{\#}.$
In this section,
we investigate some characteristics of $\mfB$ in approximating the Laplace operator.

\begin{figure}[!ht]
\centering
\epsfig{figure=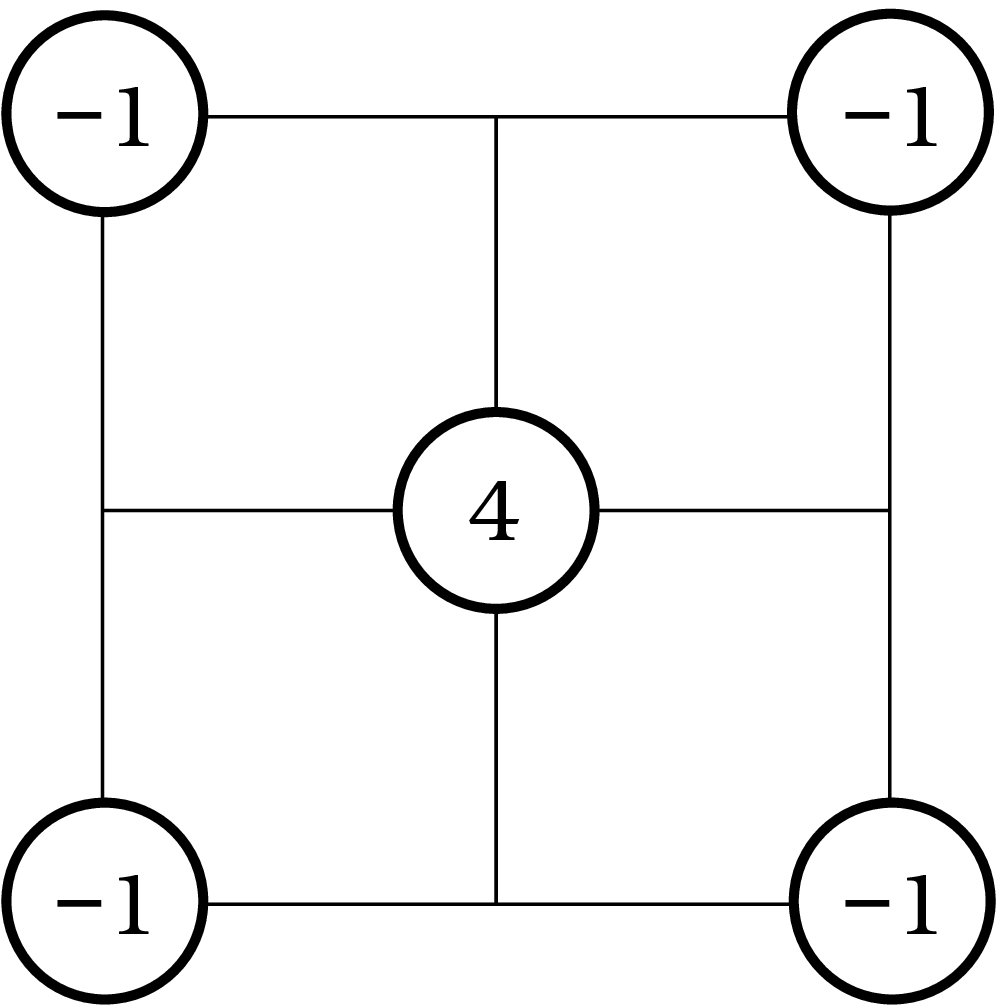,width=0.18\textwidth}
\caption{The stencil for $\bS^{\mfB}_h$ {with uniform cubes of size $h\times
  h$.}}
\label{fig:stencil-stiffness-matrix}
\end{figure}

Set $\bS^{\mfB}_h$ be the $|\mfB|\times|\mfB|$ stiffness matrix
associated with $\mfB = \{ \phi_j \}_{j=1}^{|\mfB|},$ whose components
are given by
\begin{align}\label{eq:def-stiffness-B}
(\bS^{\mfB}_h)_{jk}
=
	\sum_{K \in \Tau_h} \int_K \nabla \phi_k \cdot \nabla \phi_j \quad 1 \le j,k \le |\mfB|.
\end{align}
The local stencil for the stiffness matrix associated with $\mfB$ is shown in \figref{fig:stencil-stiffness-matrix}.
Obviously, $\bS^{\mfB}_h$ is symmetric and positive semi-definite.
The following lemma and proposition are immediate, but useful for
later uses.

\begin{lemma}\label{lem:ker-iff-constant}
Let $v_h = \sum_j v_j \phi_j$ for $\mathbf{v} = (v_j) \in \BbbR^{|\mfB|}$.
Then
$\mathbf{v} \in \ker \bS^{\mfB}_h$ if and only if
$v_h$ is a constant function in $\O$.
\end{lemma}
\begin{proposition}\label{prop:ker-decomposition}
$\ker \bS^{\mfB}_h$ can be decomposed as
\begin{align}
\ker \bS^{\mfB}_h = \ker B^{\mfB}_h \oplus \Span \mathbf{1}_{\mfB}.
\end{align}
\end{proposition}

\begin{remark}
\lemref{lem:ker-iff-constant} and
\propref{prop:ker-decomposition} are also valid in the 3D case. 
\end{remark}
Observe that
\proprefs{prop:dim-ker}{prop:ker-decomposition} directly lead to the following proposition.

\begin{proposition}\label{prop:dim-ker-stiffness}(The dimension of $\ker \bS^{\mfB}_h$)
\begin{align}
\dim \ker \bS^{\mfB}_h
= \eps(N_x)\eps(N_y) + 1.
\end{align}
\end{proposition}

\section{Numerical Schemes for Elliptic Problems with Periodic
  BC}\label{sec:numerical-scheme}
Assume that $f\in L^2(\O)$ is given such that $\int_\O f = 0.$
Consider the elliptic problem with periodic BC to find $u\in
H^1_{\#}(\O)/{\mathbb R}$ such that
$- \Delta u = f \ \text{in } \O.$
The weak formulation is as follows:
{\it find $u \in H^1_{\#}(\O)/{\mathbb R}$ such that}
\begin{equation}\label{eq:weak-formulation-periodic}
\int_\O \nabla u \cdot \nabla v 
= \int_\O f v\q \forany v \in H^1_{\#}(\O)/{\mathbb R}.
\end{equation}
By defining $a_h (u_h, v_h) := \sum_{K\in\Tau_h} \int_K \nabla u_h \cdot \nabla v_h $,
the discrete weak formulation for \eqref{eq:weak-formulation-periodic} is given as follows:
{\it find $u_h \in V^h_{\#}/{\mathbb R}$ such that}
\begin{equation}\label{eq:discrete-weak-formulation-periodic}
a_h (u_h, v_h) = \int_\O f v_h \q \forany v_h \in V^h_{\#}/{\mathbb R}.
\end{equation}

\begin{remark}
Throughout this section,
we assume that both $N_x$ and $N_y$ are even.
The other cases with odd $N_x$ and/or $N_y$ are easy to handle
owing to $V^h_{\#}=V^{\mfB,h}_{\#}$.
Also we will assume that $(\Tau_h)_{0<h}$ is a family of uniform rectangular
decomposition.
\end{remark}

\subsection*{Additional Notations \& Properties}
We compare 4 different numerical approaches to solve
\eqref{eq:discrete-weak-formulation-periodic} with the trial and test
function spaces $\mcS=\mfB^\flat, \mfB, \mfE^\flat, \mfE,$ which are
described as follows.
Due to \propref{prop:dim-ker},
we can find $\mfB^\flat$, a proper subset of $\mfB$,
which is a basis for $V^{\mfB,h}_{\#}$.
It clearly holds that $|\mfB^\flat| = \dim V^{\mfB,h}_{\#} = |\mfB| -1$.
Consider the two extended sets
$\mfE := \mfB \cup \mfA$ and
$\mfE^\flat := \mfB^\flat \cup \mfA,$
the latter of which is a basis for $V^h_{\#}$.
The characteristics of $\mfB^\flat$, $\mfB$, $\mfE^\flat$, and $\mfE$
are summarized in \tabref{tab:summary-characteristic}.
\begin{table}[!htb]
  \centering
  \begin{tabular}{ l | l | c | l }
    \hline \hline
    \multicolumn{1}{c|}{$\mcS$} & \multicolumn{1}{c|}{$|\mcS|$} & \multicolumn{1}{c|}{$\Span \mcS$} & \multicolumn{1}{c}{$\dim \Span \mcS$} \\ \hline \hline
    $\mfB^\flat$ & $N_x N_y -1$ & \multirow{2}{*}{$V^{\mfB,h}_{\#}$} & \multirow{2}{*}{$N_x N_y -1$} \\
    $\mfB$ & $N_x N_y$ & & \\ \hline
    $\mfE^\flat$ & $N_x N_y +1$ & \multirow{2}{*}{$V^h_{\#}$} & \multirow{2}{*}{$N_x N_y +1$} \\
    $\mfE$ & $N_x N_y +2$ & & \\ \hline \hline
  \end{tabular}
  \caption{Characteristics of each test and trial function set $\mcS$ when both $N_x$, $N_y$ are even}
  \label{tab:summary-characteristic}
\end{table}
For a vector $\mathbf{v}$ with $|\mfE|$ (or $|\mfE^\flat|$) number of components,
let $\mathbf{v}|_{\mfB}$ (or $\mathbf{v}|_{\mfB^\flat}$)  and $\mathbf{v}|_{\mfA}$ denote
vectors consisting of the first $|\mfB|$ (or $|\mfB^\flat|$)
components, and of the last $|\mfA|$ components, respectively.
Several properties of functions in $\mfB$ and $\mfA$ can be observed.
\begin{lemma}\label{lem:integral-property-B-A}
Let $\mfB$ and $\mfA$ be as above. Then the followings hold.
\begin{enumerate}
\item $a_h (\phi, \psi) = 0 \forany \phi \in \mfB \forany \psi \in \mfA$.
\item $a_h (\psi_\mu, \psi_\nu) = 0 \forany \psi_\mu, \psi_\nu \in \mfA$ such that $\mu \not = \nu$.
\item $\int_\O \psi = 0 \forany \psi \in \mfA$.
\item There exists an $h$-independent constant $C$ such that
$\| \psi \|_0 \le C$ and $| \psi |_{1,h} \le Ch^{-1} \forany \psi \in \mfA$.
\end{enumerate}
\end{lemma}
Now, we define a stiffness matrix associated with $\mfB^\flat$, and its variants.
Let $\bS^{\mfB^\flat}_h$ be the $|\mfB^\flat|\times |\mfB^\flat|$ stiffness matrix
associated with $\mfB^\flat$ whose components are given by
\begin{align}\label{eq:def-stiffness-B-flat}
(\bS^{\mfB^\flat}_h)_{jk}
:=
	a_h (\phi_k, \phi_j) \quad 1 \le j,k \le |\mfB^\flat|,
\end{align}
and $\mathbf{\tilde{S}}^{\mfB^\flat}_h$ be the same matrix as $\bS^{\mfB^\flat}_h$,
but the last row is modified in order to impose the zero mean value condition.
Because all the integrals $\int_\O \phi_j$ are same for all $\phi_j$ in $\mfB$,
every entry in the last row is replaced by $1$.
\begin{align}\label{eq:def-stiffness-tilde-B-flat}
(\mathbf{\tilde{S}}^{\mfB^\flat}_h)_{jk}
:=
	\begin{cases}
	a_h (\phi_k, \phi_j) & j \not = |\mfB^\flat|, \\
	1 & j = |\mfB^\flat|.
	\end{cases}
\end{align}
Note that
$\mathbf{\tilde{S}}^{\mfB^\flat}_h$ is nonsingular
whereas
both $\bS^{\mfB}_h$ and $\bS^{\mfB^\flat}_h$ are singular
with rank deficiency $2$ and $1$, respectively.
For the complementary part,
let $\bS^{\mfA}_h$ be the $|\mfA|\times|\mfA|$ stiffness matrix
associated with $\mfA$,
\begin{align}\label{eq:def-stiffness-A}
(\bS^{\mfA}_h)_{jk}
:=
	a_h (\psi_k, \psi_j) \quad 1 \le j,k \le |\mfA|.
\end{align}
Notice that $\bS^{\mfA}_h$ is a nonsingular diagonal matrix
due to \lemref{lem:integral-property-B-A}.

\subsection{Option 1: $\mcS = \mfE^\flat$ for a nonsingular nonsymmetric system}\label{sec:opt-1}
Since $\mfE^\flat$ is a basis for $V^h_{\#}$,
$\mfE^\flat$ is a natural choice as a set of trial and test functions to assemble a linear system
corresponding to \eqref{eq:discrete-weak-formulation-periodic}.
The numerical solution $u_h \in V^h_{\#}$ is uniquely expressed, associated with $\mfE^\flat$, as
\begin{align}\label{eq:sol-scheme-1}
u_h = \mathbf{\tilde{u}^\flat} \mfE^\flat
\end{align}
where $\mathbf{\tilde{u}^\flat}$ is the solution of 
the following system of equations associated with $\mfE^\flat$:
\begin{equation}\label{eq:scheme-1}
\mathcal{\tilde{L}}^{\mfE^\flat}_h \mathbf{\tilde{u}^\flat}
=
\begin{bmatrix}
\mathbf{\tilde{f}}_{\mfB^\flat} \\ \mathbf{f}_\mfA
\end{bmatrix}
\end{equation}
where
$
\mathcal{\tilde{L}}^{\mfE^\flat}_h
:=
\begin{bmatrix}
\mathbf{\tilde{S}}^{\mfB^\flat}_h & \mathbf{0} \\
\mathbf{0} & \bS^{\mfA}_h
\end{bmatrix},$
$(\mathbf{\tilde{f}}_{\mfB^\flat})_j
=
	\begin{cases}
	\int_\O f \phi_j, & 1\le j < |\mfB^\flat| \\
	0, & j = |\mfB^\flat|
	\end{cases},$ and $\mathbf{f}_\mfA=\int_\O f \mfA.
$
Due to \lemref{lem:integral-property-B-A},
$\mathcal{\tilde{L}}^{\mfE^\flat}_h$ is a block-diagonal
matrix. Moreover, it is nonsingular, but nonsymmetric
due to the modification in the last row of $\mathbf{\tilde S}^{\mfB^\flat}$
which comes from the zero mean value condition.
We can use any known numerical scheme for general linear systems, for instance GMRES,
to solve \eqref{eq:scheme-1}.
\begin{scheme}
  \noindent{\bf GMRES for $\mcS = \mfE^\flat$}
  \\
{\bf Step 1.} Take an initial vector $\bu^{(0)} \in \BbbR^{|\mfE^\flat|}.$
\\
{\bf Step 2.} Solve the nonsymmetric system \eqref{eq:scheme-1} by a
restarted GMRES
and set $\mathbf{\tilde{u}^\flat} := \bu^{(n)}$.
\\
{\bf Step 3.} The numerical solution is obtained as $u_h =
\mathbf{\tilde{u}^\flat} \mfE^\flat$.
\end{scheme}

\subsection{Option 2: $\mcS = \mfE^\flat$ for a symmetric positive semi-definite system with rank deficiency 1}\label{sec:opt-2}
In the previous approach,
the zero mean value condition is imposed in a system of equations directly.
Consequently
the associated linear system becomes nonsymmetric
due to modification of just a single row.

In this subsection we will impose the zero mean value condition indirectly
  in order to conserve symmetry of the assembled linear system.
In particular, we will impose the zero mean value condition
in post-processing stage.
  Then we can apply some fast solvers for symmetric system.
  On the other hand, nonsingularity can not be avoided any longer in this approach.
Fortunately the linear system is at least positive semi-definite.
Hence we can use a Drazin inverse as mentioned in \secref{sec:drazin-inverse}
to solve our singular system
using a Krylov iterative method under a proper condition.

Consider a system of equations for \eqref{eq:discrete-weak-formulation-periodic}
associated with $\mfE^\flat$ without any modification:
\begin{align}\label{eq:scheme-2}
\mathcal{L}^{\mfE^\flat}_h \mathbf{u^\flat}
=
\int_\O f \mfE^\flat
\quad
\text{where}
\quad
\mathcal{L}^{\mfE^\flat}_h
:=
\begin{bmatrix}
\bS^{\mfB^\flat}_h & \mathbf{0} \\
\mathbf{0} & \bS^{\mfA}_h
\end{bmatrix}.
\end{align}
Note that the above linear system is singular, and symmetric positive semi-definite.
We should find the solution $\mathbf{u^\flat}$ of the system such that
\begin{align}\label{eq:scheme-2-zero mean value}
\mathbf{u^\flat}|_{\mfB^\flat} \cdot \mathbf{1}_{\mfB^\flat} = 0
\end{align}
since
$\int_\O \mathbf{v} \mfE^\flat = 0$
if and only if
$\mathbf{v}|_{\mfB^\flat} \cdot \mathbf{1}_{\mfB^\flat} = 0$,
and the numerical solution $u^\flat_h \in V^h_{\#}$ of this scheme is obtained by
\begin{align}\label{eq:sol-scheme-2}
u^\flat_h = \mathbf{u^\flat} \mfE^\flat.
\end{align}
If a symmetric positive semi-definite system $Ax=b$ is given, 
as our formulation above,
the conjugate gradient method (CG) gives a unique Krylov solution
if the consistency condition $b\in\Im A$ holds.
The general solution is obviously obtained upto its kernel.
The kernel of the linear system \eqref{eq:scheme-2}
is closely related with the kernel of $\bS^{\mfB^\flat}_h$.
A simple analog of \secref{sec:stiffness-B} implies that
the dimension of $\ker \bS^{\mfB^\flat}_h$ is $1$, and
$\mathbf{v} \in \ker \bS^{\mfB^\flat}_h$ if and only if
$\mathbf{v} \mfB^\flat$ is a constant function in $\O$.
Note that $\mfB^\flat$ is not a partition of unity, whereas $\mfB$ is.
Let $\mathbf{w}_{\mfB^\flat}$ denote a unique vector in $\BbbR^{|\mfB^\flat|}$
such that $\mathbf{w}_{\mfB^\flat}\mfB^\flat \equiv 1$ in $\O$.
Then the kernel of the linear system in \eqref{eq:scheme-2} is simply represented by
$\Span \mathbf{w}_{\mfE^\flat}$
where
$$
\mathbf{w}_{\mfE^\flat}
=
\begin{bmatrix}
\mathbf{w}_{\mfB^\flat}^T & \mathbf{0}
\end{bmatrix}
^T
\in
\BbbR^{|\mfE^\flat|}
$$
is the trivial extension of $\mathbf{w}_{\mfB^\flat}.$
Therefore in the post-processing stage
we add a multiple of $\mathbf{w}_{\mfE^\flat}$ to the Krylov solution
to preserve \eqref{eq:scheme-2-zero mean value}.


In summary, the numerical scheme for $u^\flat_h$ is given as follows.
\begin{scheme}
\noindent{\bf CG for $\mcS = \mfE^\flat$ of rank 1 deficiency}
  \\
{\bf Step 1.} Take a vector $\bu^{(0)} \in
\BbbR^{|\mfE^\flat|}$ for an initial guess.
\\
{\bf Step 2.} Solve the singular symmetric positive semi-definite
system \eqref{eq:scheme-2} by the CG and get the
Krylov solution $\mathbf{u'} = \bu^{(n)}$.
\\
{\bf Step 3.} Add a multiple of $\mathbf{w}_{\mfE^\flat}$ to $\mathbf{u'}$ to get $\mathbf{u^\flat}$, in order to enforce \eqref{eq:scheme-2-zero mean value}, as
\begin{align*}
\mathbf{u^\flat} = \mathbf{u'} - \frac{\mathbf{u'}|_{\mfB^\flat} \cdot \mathbf{1}_{\mfB^\flat}}{\mathbf{w}_{\mfB^\flat} \cdot \mathbf{1}_{\mfB^\flat}} \mathbf{w}_{\mfE^\flat}.
\end{align*}
{\bf Step 4.} The numerical solution is obtained as $u^\flat_h = \mathbf{u^\flat} \mfE^\flat$.
\end{scheme}

\subsection{Option 3: $\mcS = \mfE$ for a symmetric positive semi-definite system with rank deficiency 2}\label{sec:opt-3}
Although symmetry and positive semi-definiteness
are key factors for an efficient numerical scheme for linear solvers,
we may not enjoy full benefits in the previous scheme.
We need the extra post-processing stage to impose the zero mean value condition.
The defect in the previous approach comes from the fact that
the Riesz representation vector for the integral functional does not belong to
the kernel of the linear system.
As shown above,
the kernel of the linear system is closely related with
the coefficient vector for the unity function.
If these two vectors coincide, we can get our solution without any post-processing stage.
The imbalance of $\mfB^\flat$ for the linear independence is also a disadvantage
to numerical implementation.

In this approach, 
we find the numerical solution $u^\natural_h \in V^h_{\#}$ such that
\begin{align}\label{eq:sol-scheme-3}
u^\natural_h = \mathbf{u^\natural} \mfE
\end{align}
where $\mathbf{u^\natural}$ is a solution of
a system of equations for \eqref{eq:discrete-weak-formulation-periodic}
associated with full $\mfE$,
\begin{align}\label{eq:scheme-3}
\mathcal{L}^{\mfE}_h \mathbf{u^\natural}
:=
\begin{bmatrix}
\bS^{\mfB}_h & \mathbf{0} \\
\mathbf{0} & \bS^{\mfA}_h
\end{bmatrix}
\mathbf{u^\natural}
=
\int_\O f \mfE
\end{align}
with
\begin{align}\label{eq:scheme-3-zero mean value}
\mathbf{u^\natural}|_{\mfB} \cdot \mathbf{1}_{\mfB} = 0,
\end{align}
since
$\int_\O \mathbf{v} \mfE = 0$
if and only if
$\mathbf{v}|_{\mfB} \cdot \mathbf{1}_{\mfB} = 0$.
The numerical solution $u^\natural$ is unique because
a solution of the linear system is unique 
upto an additive nontrivial representation for the zero function in $\mfB$.
We want to emphasize that,
unlike the previous scheme,
$\mathbf{1}_{\mfB}$ belongs to the kernel of $\bS^{\mfB}_h$
as shown in \eqref{prop:ker-decomposition}.
It implies that, without any extra post-processing stage,
we can find the solution of the linear system which satisfies the zero mean value condition \eqref{eq:scheme-3-zero mean value}
if an initial guess is chosen to satisfy the same condition.

In summary, we have the numerical solution $u^\natural_h$ as follows.

\begin{scheme}
  \noindent{\bf CG for $\mcS = \mfE$ of rank 2 deficiency}
  \\
{\bf Step 1.} Take an initial vector $\bu^{(0)} \in \BbbR^{|\mfE|}$ which satisfies
$\bu^{(0)}|_{\mfB} \cdot \mathbf{1}_{\mfB} = 0$.
\\
{\bf Step 2.} Solve the singular symmetric positive semi-definite
system \eqref{eq:scheme-3} by the CG and get the Krylov solution
$\mathbf{u^\natural} = \bu^{(n)}$.
\\
{\bf Step 3.} The numerical solution is obtained as $u^\natural_h = \mathbf{u^\natural} \mfE$.
\end{scheme}

\subsection{Option 4: $\mcS = \mfB$ for a symmetric positive semi-definite system with rank deficiency 2}\label{sec:opt-4}
Consider a system of equations associated only with $\mfB$ for
\eqref{eq:discrete-weak-formulation-periodic} to find
$ \mathbf{\bar{u}^\natural} \in V^{\mfB,h}_{\#}$ such that
\begin{align}\label{eq:scheme-4}
\mathcal{L}^{\mfB}_h \mathbf{\bar{u}^\natural}
:=
\bS^{\mfB}_h \mathbf{\bar{u}^\natural}
=
\int_\O f \mfB.
\end{align}
Starting from an initial vector $\bu^{(0)} \in \BbbR^{|\mfB|}$ which satisfies
$\bu^{(0)} \cdot \mathbf{1}_{\mfB} = 0$,
let $\mathbf{\bar{u}^\natural}$ be the Krylov solution of the linear system.
The numerical solution $\bar{u}^\natural_h \in V^{\mfB,h}_{\#}$ is obtained by
\begin{align}\label{eq:sol-scheme-4}
\bar{u}^\natural_h = \mathbf{\bar{u}^\natural} \mfB.
\end{align}
We summarize the above procedure as follows.
\begin{scheme}
  \noindent{\bf CG for $\mcS = \mfB$ of rank 2 deficiency}
  \\
{\bf Step 1.} Take an initial vector $\bu^{(0)} \in \BbbR^{|\mfE|}$ which satisfies
$\bu^{(0)}|_{\mfB} \cdot \mathbf{1}_{\mfB} = 0$.
\\
{\bf Step 2.} Solve the singular symmetric positive semi-definite
system \eqref{eq:scheme-4} by the CG and get the Krylov solution
$\mathbf{\bar{u}^\natural} = \bu^{(n)}$.
\\
{\bf Step 3.} The numerical solution is obtained as $\bar{u}^\natural_h = \mathbf{\bar{u}^\natural} \mfB$.
\end{scheme}

\subsection*{Main Theorem: Relation Between Numerical Solutions}
The following theorem states the relation between all numerical solutions discussed above.

\begin{theorem}\label{thm:numerical-sols}
{Let $(\Tau_h)_{0<h}$ be a family of uniform rectangular
  decomposition, that is, $\Tau_h =\widetilde\Tau_h$ for all $h.$
Assume that $N_x$ and $N_y$ are even.}
Let $u_h$, $u^\flat_h$, $u^\natural_h$, $\bar{u}^\natural_h$ be the numerical solutions of
\eqref{eq:weak-formulation-periodic}
as \eqref{eq:sol-scheme-1}, \eqref{eq:sol-scheme-2}, \eqref{eq:sol-scheme-3}, \eqref{eq:sol-scheme-4}, respectively.
Then
$u_h = u^\flat_h = u^\natural_h$, and
\begin{equation*}
\|u^\natural_h - \bar{u}^\natural_h\|_0
\le C
	h^2 \|f\|_0, \
|u^\natural_h - \bar{u}^\natural_h|_{1,h}
\le C
	h \|f\|_0.
\end{equation*}
\end{theorem}

\begin{remark}
  \thmref{thm:numerical-sols} provides theoretical error bounds
  with a set of test and trial functions which are redundant but
  easy to implement, instead of a set of exact solutions which are exactly
  fitted but complicated to implement.
\end{remark}

\begin{proof}
Let $\mathbf{u^\flat}$ and $\mathbf{\tilde{u}^\flat}$ be the solutions as in \secrefs{sec:opt-1}{sec:opt-2}, respectively.
Note that two linear systems \eqref{eq:scheme-1} and \eqref{eq:scheme-2}
coincide except $|\mfB^\flat|$-th row.
Even on $|\mfB^\flat|$-th row,
\begin{align*}
\left(\mathcal{\tilde{L}}^{\mfE^\flat}_h \mathbf{u^\flat} \right)_{|\mfB^\flat|}
&=
	\mathbf{1}_{\mfB^\flat} 
	\cdot 
	\left.\left( \mathbf{u'} 
	- \frac{\mathbf{u'}|_{\mfB^\flat} \cdot \mathbf{1}_{\mfB^\flat}}{\mathbf{w}_{\mfB^\flat} \cdot \mathbf{1}_{\mfB^\flat}}
	\mathbf{w}_{\mfE^\flat} \right)\right|_{\mfB^\flat} \\
&=
	\mathbf{1}_{\mfB^\flat} 
	\cdot 
	\left( \mathbf{u'}|_{\mfB^\flat} 
	- \frac{\mathbf{u'}|_{\mfB^\flat} \cdot \mathbf{1}_{\mfB^\flat}}{\mathbf{w}_{\mfB^\flat} \cdot \mathbf{1}_{\mfB^\flat}} 
	\mathbf{w}_{\mfB^\flat} \right) \\
&=
	0.
\end{align*}
Thus
$
\mathcal{\tilde{L}}^{\mfE^\flat}_h \mathbf{u^\flat}
=
	\begin{bmatrix}
	\mathbf{\tilde{f}}_{\mfB^\flat} \\ \mathbf{f}_\mfA
	\end{bmatrix}
=
	\mathcal{\tilde{L}}^{\mfE^\flat}_h \mathbf{\tilde{u}^\flat}
$,
and it implies $\mathbf{u^\flat} = \mathbf{\tilde{u}^\flat}$ 
because $\mathcal{\tilde{L}}^{\mfE^\flat}_h$ is nonsingular.
It concludes $u_h = u^\flat_h$.

Let $\mathbf{u^\natural}$ be the solution as in \secref{sec:opt-3}. Then,
we have $\mathbf{u^\natural}|_{\mfA} = \mathbf{u^\flat}|_{\mfA}.$
Let $\begin{bmatrix} \mathbf{u^\flat}|_{\mfB^\flat} \\ 0 \end{bmatrix}$ be
a trivial extension of $\mathbf{u^\flat}|_{\mfB^\flat}$ into a vector in $\BbbR^{|\mfB|}$
by padding a single zero.
Note that
$\sum_{j=1}^{|\mfB|} (\bS^{\mfB}_h)_{jk} = 0$ for all $1\le k \le |\mfB|$.
Due to the definition of $\mathbf{u^\flat}$, we have
\begin{align*}
\bS^{\mfB}_h \begin{bmatrix} \mathbf{u^\flat}|_{\mfB^\flat} \\ 0 \end{bmatrix}
&=
	\begin{bmatrix}
	\bS^{\mfB^\flat}_h \mathbf{u^\flat}|_{\mfB^\flat} \\
	[\bS^{\mfB}_h]_{|\mfB|,1:|\mfB^\flat|} \mathbf{u^\flat}|_{\mfB^\flat}
	\end{bmatrix}
=
	\begin{bmatrix}
	\int_\O f \mfB^\flat \\
	- \sum_{j \not = |\mfB|} \limits [\bS^{\mfB}_h]_{j, 1:|\mfB^\flat|} \mathbf{u^\flat}|_{\mfB^\flat}
	\end{bmatrix} \\
&=
	\begin{bmatrix}
	\int_\O f \mfB^\flat \\
	- \sum_{j=1}^{|\mfB^\flat|} [\bS^{\mfB^\flat}_h]_{j, 1:|\mfB^\flat|} \mathbf{u^\flat}|_{\mfB^\flat}
	\end{bmatrix}
=
	\begin{bmatrix}
	\int_\O f \mfB^\flat \\
	- \sum_{j=1}^{|\mfB^\flat|} \int_\O f \phi_j
      \end{bmatrix}
  \\
&=
	\begin{bmatrix}
	\int_\O f \mfB^\flat \\
	\int_\O f (\phi_{|\mfB|}-1)
	\end{bmatrix}
=
	\int_\O f \mfB,
\end{align*}
since $\mfB$ is a partition of unity and $\int_\O f = 0$.
On the other hand,
the definition of $\mathbf{u^\natural}$ implies
$\bS^{\mfB}_h \mathbf{u^\natural}|_{\mfB} = \int_\O f \mfB$.
Thus
$\mathbf{u^\natural}|_{\mfB} - \begin{bmatrix} \mathbf{u^\flat}|_{\mfB^\flat} \\ 0 \end{bmatrix}$ 
is in the kernel of $\bS^{\mfB}_h$, which is decomposed as \propref{prop:ker-decomposition}.
Due to the zero mean value condition in each scheme, 
$\left( \mathbf{u^\natural}|_{\mfB} - \begin{bmatrix} \mathbf{u^\flat}|_{\mfB^\flat} \\ 0 \end{bmatrix} \right) \cdot \mathbf{1}_{\mfB} = \mathbf{u^\natural}|_{\mfB} \cdot \mathbf{1}_{\mfB} - \mathbf{u^\flat}|_{\mfB^\flat} \cdot \mathbf{1}_{\mfB^\flat} = 0$.
Therefore
$
\mathbf{u^\natural}|_{\mfB} - \begin{bmatrix} \mathbf{u^\flat}|_{\mfB^\flat} \\ 0 \end{bmatrix}
$
must belong to $\ker B^{\mfB}_h$,
and consequently
$\left( \mathbf{u^\natural}|_{\mfB} - \begin{bmatrix} \mathbf{u^\flat}|_{\mfB^\flat} \\ 0 \end{bmatrix} \right) \mfB = \mathbf{u^\natural}|_{\mfB} \mfB - \mathbf{u^\flat}|_{\mfB^\flat} \mfB^\flat$ is equal to $0$.
This implies $u^\natural_h = u^\flat_h$.

Let $\mathbf{\bar{u}^\natural}$ be the solution as in \secref{sec:opt-4}.
Note that $\mathbf{u^\natural}|_{\mfB} = \mathbf{\bar{u}^\natural}$, and
\begin{align*}
\mathbf{u^\natural}|_{\mfA}
=
	\operatorname{diag}\left(	a_h(\psi_x,\psi_x),a_h(\psi_y,\psi_y)\right)^{-1}
	\int_\O f \mfA
\le C
	h^2 \int_\O f \mfA
\end{align*}
due to \lemref{lem:integral-property-B-A}.
Owing to
$
\int_\O f \psi
\le C 
	\left(\int_\O |f|^2\right)^{1/2}
	\left(\int_\O |\psi|^2\right)^{1/2} 
\le C 
\|f\|_0 \forany \psi \in \mfA,
$
each component of $\mathbf{u^\natural}|_{\mfA}$
is bounded by $\mathcal{O}(h^2)$.
Hence we have desired estimates the difference between $u^\natural_h$ and $\bar{u}^\natural_h$
in $L^2$- and $H^1$-(semi-)norm.
\end{proof}

\subsection{Numerical results}
For the scheme Option 1,
we use the restarted GMRES scheme in MGMRES library
provided by Ju and Burkardt \cite{ju-burkardt-mgmres}.
We emphasize that
we replace one of essentially linearly dependent rows of $\bS^{\mfB^\flat}_h$
by the zero mean value condition in order to make $\mathbf{\tilde{S}}^{\mfB^\flat}_h$ nonsingular.

\begin{example}\label{ex:1}
{Consider the problem} \eqref{eq:weak-formulation-periodic} on the domain $\O=(0,1)^2$
with the exact solution $u(x,y) = s(x) s(y)$
where
$s(t) = \sum_{k=1}^{3} \frac{4}{(2k-1)\pi} \sin \Big(2(2k-1)\pi t\Big),$
a truncated Fourier series for the square wave.
\end{example}
For each option,
the error in energy norm and $L^2$-norm
for Example \ref{ex:1}
are shown in \tabref{tab:ex-truncated-fourier-square-wave}.
We observe that all schemes give very similar numerical solutions.

\begin{example}\label{ex:2}
{Consider the same problem as in Example \ref{ex:1}} with the exact solution $u(x,y) = s(x) s(y)$
where
$s(t) = \exp \left(-\frac{1}{1-(2t-1)^2}\right) t^2 (1-t) + C,$
with a constant $C$ satisfying $\int_{[0,1]} s = 0$.
\end{example}
\tabref{tab:ex-bump-times-smooth} shows numerical results for Example \ref{ex:2}
in each option,
and all options give almost the same result, as the previous example.
The iteration number and elapsed time in each option in the case of $h=1/256$ 
are shown in \tabref{tab:iter-time}.
We observe decrease of the iteration number and elapsed time
in the option 3 compared to the option 2.
Decrease from the option 3 to the option 4 is quite natural
because we only use the node--based functions as trial and test functions for the option 4.

\begin{table}[t]
  \scriptsize
  \centering
  \begin{tabular}{ l | c  c  c  c | c  c  c  c }
    \hline \hline
    & \multicolumn{4}{c|}{Opt 1} & \multicolumn{4}{c}{Opt 2} \\ \cline{2-9} 
    \multicolumn{1}{c|}{$h$} & $|u-u_h|_{1,h}$ & order & $\|u-u_h\|_0$ & order & $|u-u^\flat_h|_{1,h}$ & order & $\|u-u^\flat_h\|_0$ & order \\ \hline
    1/8 & 1.123E+01 & - & 4.230E-01 & - & 1.123E+01 & - & 4.230E-01 & - \\
    1/16 & 5.466E-00 & 1.039 & 8.607E-02 & 2.297 & 5.466E-00 & 1.039 & 8.607E-02 & 2.297 \\
    1/32 & 2.832E-00 & 0.949 & 2.216E-02 & 1.957 & 2.832E-00 & 0.949 & 2.216E-02 & 1.957 \\
    1/64 & 1.429E-00 & 0.987 & 5.585E-03 & 1.989 & 1.429E-00 & 0.987 & 5.585E-03 & 1.989 \\
    1/128 & 7.160E-01 & 0.997 & 1.399E-03 & 1.997 & 7.160E-01 & 0.997 & 1.399E-03 & 1.997 \\
    1/256 & 3.582E-01 & 0.999 & 3.499E-04 & 1.999 & 3.582E-01 & 0.999 & 3.499E-04 & 1.999 \\
  \end{tabular}

  \begin{tabular}{ l | c  c  c  c | c  c  c  c }
    \hline \hline
    & \multicolumn{4}{c|}{Opt 3} & \multicolumn{4}{c}{Opt 4} \\ \cline{2-9} 
    \multicolumn{1}{c|}{$h$} & $|u-u^\natural_h|_{1,h}$ & order & $\|u-u^\natural_h\|_0$ & order & $|u-\bar{u}^\natural_h|_{1,h}$ & order & $\|u-\bar{u}^\natural_h\|_0$ & order \\ \hline
    1/8 & 1.123E+01 & - & 4.230E-01& - & 1.123E+01 & - & 4.230E-01 & - \\
    1/16 & 5.466E-00 & 1.039 & 8.607E-02 & 2.297 & 5.466E-00 & 1.039 & 8.607E-02 & 2.297 \\
    1/32 & 2.832E-00 & 0.949 & 2.216E-02 & 1.957 & 2.832E-00 & 0.949 & 2.216E-02 & 1.957 \\
    1/64 & 1.429E-00 & 0.987 & 5.585E-03 & 1.989 & 1.429E-00 & 0.987 & 5.585E-03 & 1.989 \\
    1/128 & 7.160E-01 & 0.997 & 1.399E-03 & 1.997 & 7.160E-01 & 0.997 & 1.399E-03 & 1.997 \\
    1/256 & 3.582E-01 & 0.999 & 3.499E-04 & 1.999 & 3.582E-01 & 0.999 & 3.499E-04 & 1.999 \\
    \hline \hline
  \end{tabular}
    \caption{{Numerical results for Example \ref{ex:1}.}}
  \label{tab:ex-truncated-fourier-square-wave}
\end{table}

\begin{table}[t]
  \scriptsize
  \centering
  \begin{tabular}{ l | c  c  c  c | c  c  c  c }
    \hline \hline
    & \multicolumn{4}{c|}{Opt 1} & \multicolumn{4}{c}{Opt 2} \\ \cline{2-9} 
    \multicolumn{1}{c|}{$h$} & $|u-u_h|_{1,h}$ & order & $\|u-u_h\|_0$ & order & $|u-u^\flat_h|_{1,h}$ & order & $\|u-u^\flat_h\|_0$ & order \\ \hline
    1/8 & 1.225E-03 & - & 5.649E-05 & - & 1.225E-03 & - & 5.649E-05 & - \\
    1/16 & 6.024E-04 & 1.024 & 1.033E-05 & 2.450 & 6.024E-04 & 1.024 & 1.033E-05 & 2.450 \\
    1/32 & 3.045E-04 & 0.984 & 1.949E-06 & 2.406 & 3.045E-04 & 0.984 & 1.949E-06 & 2.406 \\
    1/64 & 1.527E-04 & 0.996 & 4.682E-07 & 2.058 & 1.527E-04 & 0.996 & 4.682E-07 & 2.058 \\
    1/128 & 7.642E-05 & 0.999 & 1.171E-07 & 1.999 & 7.642E-05 & 0.999 & 1.171E-07 & 1.999 \\
    1/256 & 3.822E-05 & 1.000 & 2.929E-08 & 2.000 & 3.822E-05 & 1.000 & 2.929E-08 & 2.000 \\
  \end{tabular}

  \begin{tabular}{ l | c  c  c  c | c  c  c  c }
    \hline \hline
    & \multicolumn{4}{c|}{Opt 3} & \multicolumn{4}{c}{Opt 4} \\ \cline{2-9} 
    \multicolumn{1}{c|}{$h$} & $|u-u^\natural_h|_{1,h}$ & order & $\|u-u^\natural_h\|_0$ & order & $|u-\bar{u}^\natural_h|_{1,h}$ & order & $\|u-\bar{u}^\natural_h\|_0$ & order \\ \hline
    1/8 & 1.225E-03 & - & 5.649E-05 & - & 1.225E-03 & - & 5.649E-05 & - \\
    1/16 & 6.024E-04 & 1.024 & 1.033E-05 & 2.450 & 6.024E-04 & 1.024 & 1.033E-05 & 2.450 \\
    1/32 & 3.045E-04 & 0.984 & 1.949E-06 & 2.406 & 3.045E-04 & 0.984 & 1.949E-06 & 2.406 \\
    1/64 & 1.527E-04 & 0.996 & 4.682E-07 & 2.058 & 1.527E-04 & 0.996 & 4.682E-07 & 2.058 \\
    1/128 & 7.642E-05 & 0.999 & 1.171E-07 & 1.999 & 7.642E-05 & 0.999 & 1.171E-07 & 1.999 \\
    1/256 & 3.822E-05 & 1.000 & 2.929E-08 & 2.000 & 3.822E-05 & 1.000 & 2.929E-08 & 2.000 \\
    \hline \hline
  \end{tabular}
      \caption{{Numerical results for Example \ref{ex:2}.}}
  \label{tab:ex-bump-times-smooth}
\end{table}

\begin{table}[h!]
  \scriptsize
  \centering
  \begin{tabular}{ c | c  c  c }
    \hline \hline
    & solver & iter & time (sec.) \\ \hline
    Opt 1 & GMRES(20) & 4944 & 61.52 \\
    Opt 2 & CG & 817 & 3.30 \\
    Opt 3 & CG & 437 & 1.80 \\
    Opt 4 & CG & 318 & 1.33 \\
    \hline \hline
  \end{tabular}
  \caption{Iteration number and elapsed time in each option with
    $256\times 256$ mesh}
  \label{tab:iter-time}
\end{table}

\section{Extension to the 3D Case}\label{sec:3d}
In this section we consider the case of $d=3.$

\subsection{Dimension of finite element spaces  in 3D}
The following lemma is the 3D analog of \lemref{lem:dim-formula-2d}.
\begin{lemma}
For $\O \subset \BbbR^3$, we have
$$
\dim(V^h)
    =
	\#(\text{faces}) - 2 \#(\text{cells}) 
	- \#(\text{minimally essential discrete BCs}).
$$
\end{lemma}
\begin{proof}
We can rewrite the dice rule in a single 3D cubic cell $K\in\Tau_h$ into two separated relations:
\begin{align*}
v_h(m_1^K) - v_h(m_2^K) + v_h(m_6^K) - v_h(m_5^K) &= 0, \\
v_h(m_1^K) - v_h(m_3^K) + v_h(m_6^K) - v_h(m_4^K) &= 0
\end{align*}
for all $v_h\in V^h$
where $m_j^K$ is the barycenter of face $f_j^K$ of $K$,
and the faces are arranged to satisfy that
the sum of indices in opposite faces is equal to $7$, as an ordinary dice.
Since each relation reduces the number of DOFs in the
finite element space by $1$, same as in the 2D case,
the claim is derived in consequence.
\end{proof}

\begin{proposition}(Neumann and Dirichlet BCs in 3D)
\begin{align*}
& \#(\text{minimally essential discrete BCs}) \\
&\quad =
	\begin{cases}
	0 & \text{for Neumann BC}, \\
	\!\begin{aligned}
	& 2 (N_x N_y + N_y N_z + N_z N_x) \\
	& \qquad
	-
	(N_x + N_y + N_z)
	+
	1
	\end{aligned}
	& \text{with homogeneous Dirichlet BC}
	\end{cases}
\end{align*}
Consequently,
\begin{subequations}
\begin{align}
& \dim V^h = N_V-1, \\  
& \dim V^h_0 = N_V^i.  
\end{align}
\end{subequations}
\end{proposition}
\begin{proof}
It is enough to consider the homogeneous Dirichlet boundary case
since there is nothing to prove in the Neumann case.
Suppose that the homogeneous Dirichlet BC is given.
Similarly to the argument in 2D,
we need to investigate induced relations on boundary barycenter values.
Consider the {topological} $x$-direction first,
and classify all cells into $N_x$ groups
by their position in $x$.
Then each group consists of $N_y \times N_z$ cells
which are attached in the {topological} $y$- and $z$-directions.
For each cell in a group,
the dice rule in 3D implies a relation between 4 barycenter values on 4 faces
such that each of them is parallel to the {topological} $xy$- or $zx$-plane.
Similarly to the 2D case,
a collection of such relations from all cells in a group derives
a single relation consisting of an alternating sum of $2N_y + 2N_z$ barycenter values on a set of boundary faces
and it will be called a {\em strip} perpendicular to the {topological} $x$-axis.
This induced relation on the strip is well-defined
because the number of faces in the strip is always even.
\figref{fig:cube-strip} shows an example of a strip perpendicular to
the {topological} $x$-axis.
The signs on the strip represent the alternating sum of boundary barycenter values.
For the {topological} $x$-direction, there are $N_x$ strips perpendicular to the {topological} $x$-axis,
and corresponding relations between barycenters on boundary faces.
Repeating similar arguments for the {topological} $y$- and $z$-directions,
we can find totally $N_x + N_y + N_z$ strips and corresponding relations between boundary barycenters.

\begin{figure}[!ht]
\centering
\epsfig{figure=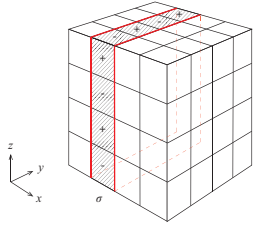,width=0.25\textwidth,height=0.14\textheight}
 \caption{An example of a strip}
 \label{fig:cube-strip}
\end{figure}

However, these induced relations are linearly independent.
Choose an element $K$ from one of corners in $\Tau_h$.
There are three strips $\sigma^x_K$, $\sigma^y_K$, $\sigma^z_K$ which are attached to $K$,
and {topologically} perpendicular to the $x$-, $y$-, $z$-axes, respectively.
Let us call each of these strips {\em the standard strip} for each axis.
There are two options to assign proper alternating signs to barycenter values on each standard strip
in order to make a corresponding alternating relation between boundary barycenters.
For each standard strip,
we choose an option for alternating sign in the relation
to cancel out all boundary barycenters which belong to $K$
when summing up all three relations on three standard strips.
We will call them {\em the standard choices}.
Consider $\sigma$, a strip among others, which is obviously parallel to one of these standard strips,
without loss of generality, $\sigma^x_K$. 
There are also two options for alternating sign in the relation on $\sigma$.
One option is same to the standard choice on $\sigma^x_K$:
in this option,
the sign for each boundary barycenter on $\sigma$ is equal to
the sign for the corresponding boundary barycenter in the standard choice on $\sigma^x_K$.
The other option is just opposite to the standard choice.
We make a choice on $\sigma$ depending on the distance from $\sigma^x_K$.
If $\sigma$ is adjacent to $\sigma^x_K$,
or is away from $\sigma^x_K$ by an even number of faces in {the topological} $x$-direction,
then we choose an option for alternating sign on $\sigma$ to be opposite to the standard choice on $\sigma^x_K$.
If $\sigma$ is away from $\sigma^x_K$ by an odd number of faces in {the topological} $x$-direction,
then the same alternating sign as the standard choice is chosen on $\sigma$.
Under this rule, we can make all choices for alternating sign in the induced relations on all $N_x + N_y + N_z$ strips.
And it can be easily shown that
the sum of all induced relations on all strips with chosen alternating sign becomes a trivial relation.
It implies that there is a single linear relation between those induced relations on all strips.
Therefore,
\begin{align*}
&\#(\text{minimally essential discrete BCs}) \\
& \qquad =
	\#(\text{boundary faces}) - \#(\text{independent relations}) \\
& \qquad =
	2 ( N_x N_y + N_y N_z + N_z N_x)
	-
	(N_x + N_y + N_z -1).
\end{align*}
\end{proof}
Depending on the evenness of $N_x,N_y,$ and $N_z,$ we have the
following result on the dimension of periodic finite element space.
\begin{proposition}(Periodic BC in 3D)\label{prop:dim-periodic-3d}
In the case of periodic BC, we have
\begin{align*}
& \#(\text{minimally essential discrete BCs}) \\
& \quad
=
	(N_x N_y + N_y N_z + N_z N_x) \nonumber \\
	& \qquad -
	\left[N_x \eps({N_y}) \eps({N_z})+ N_y \eps({N_x}) \eps({N_z}) + N_z \eps({N_x}) \eps({N_y})
	- \eps({N_x}) \eps({N_y}) \eps({N_z})\right],
\end{align*}
and
\begin{eqnarray*}
\dim V^h_{\#}
&=N_x N_y N_z -	\left[N_x \eps({N_y}) \eps({N_z})+ N_y \eps({N_x})
                \eps({N_z}) + N_z \eps({N_x}) \eps({N_y}) \right. \nonumber \\
&\qquad\qquad	\left.- \eps({N_x}) \eps({N_y}) \eps({N_z})\right].
\end{eqnarray*}
\end{proposition}
\begin{proof}
Due to the same reason discussed in the 2D case,
an induced relation between boundary barycenter values on a strip
perpendicular to the $x$-axis
can help to impose the periodic BC
only when both $N_y$ and $N_z$ are even.
In this case, coincidence of two barycenter values of the last boundary face pair
is naturally achieved by pairwise coincidence of barycenter values of other boundary face pairs in the strip.
Consequently,
totally $N_x$ periodic BCs can hold naturally
due to other periodic BCs and induced boundary relations on strips
perpendicular to the $x$-axis.
Similar claims hold for induced boundary relations on strips {topologically} perpendicular to $y$-, and $z$-directional axes.

However, as discussed in the case of Dirichlet BC,
due to the linear dependence between $N_x + N_y + N_z$ induced relations on all strips
we have to consider $1$ redundant relation when all $N_x + N_y + N_z$
strips are taken into account of
{\it i.e.},
all $N_x$, $N_y$ and $N_z$ are even.
This completes the proof.
\end{proof}

\subsection{Linear dependence of $\mfB$ in 3D}
In this section,
we identify a global coefficient representation for node--based functions in $\mfB$
with a vector in $\BbbR^{|\mfB|}$.
With this identification, 
we use a vector $\mathbf{c} \in \BbbR^{|\mfB|}$
to represent a global coefficient representation on given 3D grid $\Tau_h$.
In this sense, we denote the local coefficients of $\mathbf{c}$ in $Q \in \Tau_h$ by
$\left.\mathbf{c}\right|_{Q}$.
For the sake of simple description,
we use this abusive notation as long as there is no chance of misunderstanding.
A surjective linear map $B^{\mfB}_h:\BbbR^{|\mfB|} \rightarrow V^{\mfB,h}_{\#}$
defined in \secref{sec:basis-finite-space-periodic} is obviously
extended to the 3D case.

\begin{figure}[!ht]
\centering
\setlength{\unitlength}{2.5mm}
\begin{picture}(8,12)(0,-2)
\thinlines
\multiput(2,2)(2,2){1}{\line(1,0){6}}
\multiput(0,0)(6,0){1}{\line(1,1){2}}
\multiput(2,2)(6,0){1}{\line(0,1){6}}
\thicklines
\multiput(0,0)(2,2){1}{\line(1,0){6}}
\multiput(6,0)(6,0){1}{\line(1,1){2}}
\multiput(0,6)(2,2){2}{\line(1,0){6}}
\multiput(0,6)(6,0){2}{\line(1,1){2}}
\multiput(0,0)(6,0){2}{\line(0,1){6}}
\multiput(0,0)(6,0){2}{\line(0,1){6}}
\multiput(8,2)(6,0){1}{\line(0,1){6}}
\scriptsize
\put(0, 0){\makebox(0.0,0.0)[tr]{$-1$}}
\put(6, 0){\makebox(0.0,0.0)[tr]{$+1$}}
\put(2, 2){\makebox(0.0,0.0)[bl]{$+1$}}
\put(8, 2){\makebox(0.0,0.0)[bl]{$-1$}}
\put(0, 6){\makebox(0.0,0.0)[tr]{$+1$}}
\put(6, 6){\makebox(0.0,0.0)[tr]{$-1$}}
\put(2, 8){\makebox(0.0,0.0)[bl]{$-1$}}
\put(8, 8){\makebox(0.0,0.0)[bl]{$+1$}}
\put(3, -1.5){\makebox(0.0,0.0)[c]{$\mcA$}}
\end{picture}
\hspace{4cm}
\begin{picture}(8,12)(0,-2)
\thinlines
\multiput(0,0)(6,0){1}{\line(1,1){2}}
\multiput(2,2)(6,0){1}{\line(0,1){6}}
\thicklines
\multiput(6,0)(6,0){1}{\line(1,1){2}}
\multiput(0,6)(6,0){2}{\line(1,1){2}}
\multiput(0,0)(6,0){2}{\line(0,1){6}}
\multiput(0,0)(6,0){2}{\line(0,1){6}}
\multiput(8,2)(6,0){1}{\line(0,1){6}}
\scriptsize
\put(0, 0){\makebox(0.0,0.0)[tr]{$-1$}}
\put(6, 0){\makebox(0.0,0.0)[tr]{$-1$}}
\put(2, 2){\makebox(0.0,0.0)[bl]{$+1$}}
\put(8, 2){\makebox(0.0,0.0)[bl]{$+1$}}
\put(0, 6){\makebox(0.0,0.0)[tr]{$+1$}}
\put(6, 6){\makebox(0.0,0.0)[tr]{$+1$}}
\put(2, 8){\makebox(0.0,0.0)[bl]{$-1$}}
\put(8, 8){\makebox(0.0,0.0)[bl]{$-1$}}
{\color{red}
\thicklines
\multiput(0,0)(2,2){2}{\line(1,0){6}}
\multiput(0,6)(2,2){2}{\line(1,0){6}}
}
\put(3, -1.5){\makebox(0.0,0.0)[c]{$\mcX$}}
\end{picture}
\\
\begin{picture}(2,2)(0,0)
\thicklines
\put(0,0){\vector(1,0){3}}
\put(3,-0.2){\makebox(0.0,0.0)[t]{$x$}}
\put(0,0){\vector(1,1){1.5}}
\put(1.7,1.7){\makebox(0.0,0.0)[bl]{$y$}}
\put(0,0){\vector(0,1){3}}
\put(-0.2,3){\makebox(0.0,0.0)[r]{$z$}}
\end{picture}
\\
\begin{picture}(8,10)(0,-2)
\thinlines
\multiput(2,2)(2,2){1}{\line(1,0){6}}
\multiput(2,2)(6,0){1}{\line(0,1){6}}
\thicklines
\multiput(0,0)(2,2){1}{\line(1,0){6}}
\multiput(0,6)(2,2){2}{\line(1,0){6}}
\multiput(0,0)(6,0){2}{\line(0,1){6}}
\multiput(0,0)(6,0){2}{\line(0,1){6}}
\multiput(8,2)(6,0){1}{\line(0,1){6}}
\scriptsize
\put(0, 0){\makebox(0.0,0.0)[tr]{$-1$}}
\put(6, 0){\makebox(0.0,0.0)[tr]{$+1$}}
\put(2, 2){\makebox(0.0,0.0)[bl]{$-1$}}
\put(8, 2){\makebox(0.0,0.0)[bl]{$+1$}}
\put(0, 6){\makebox(0.0,0.0)[tr]{$+1$}}
\put(6, 6){\makebox(0.0,0.0)[tr]{$-1$}}
\put(2, 8){\makebox(0.0,0.0)[bl]{$+1$}}
\put(8, 8){\makebox(0.0,0.0)[bl]{$-1$}}
{\color{red}
\thicklines
\multiput(0,0)(6,0){2}{\line(1,1){2}}
\multiput(0,6)(6,0){2}{\line(1,1){2}}
}
\put(3, -1.5){\makebox(0.0,0.0)[c]{$\mcY$}}
\end{picture}
\hspace{4cm}
\begin{picture}(8,10)(0,-2)
\thinlines
\multiput(2,2)(2,2){1}{\line(1,0){6}}
\multiput(0,0)(6,0){1}{\line(1,1){2}}
\thicklines
\multiput(0,0)(2,2){1}{\line(1,0){6}}
\multiput(6,0)(6,0){1}{\line(1,1){2}}
\multiput(0,6)(2,2){2}{\line(1,0){6}}
\multiput(0,6)(6,0){2}{\line(1,1){2}}
\multiput(8,2)(6,0){1}{\line(0,1){6}}
\scriptsize
\put(0, 0){\makebox(0.0,0.0)[tr]{$-1$}}
\put(6, 0){\makebox(0.0,0.0)[tr]{$+1$}}
\put(2, 2){\makebox(0.0,0.0)[bl]{$+1$}}
\put(8, 2){\makebox(0.0,0.0)[bl]{$-1$}}
\put(0, 6){\makebox(0.0,0.0)[tr]{$-1$}}
\put(6, 6){\makebox(0.0,0.0)[tr]{$+1$}}
\put(2, 8){\makebox(0.0,0.0)[bl]{$+1$}}
\put(8, 8){\makebox(0.0,0.0)[bl]{$-1$}}
{\color{red}
\thicklines
\multiput(0,0)(6,0){2}{\line(0,1){6}}
\multiput(2,2)(6,0){2}{\line(0,1){6}}
}
\put(3, -1.5){\makebox(0.0,0.0)[c]{$\mcZ$}}
\end{picture}
\caption{Nontrivial representations for the zero function in an element: $\mcA$, $\mcX$, $\mcY$, $\mcZ$}
\label{fig:nontrivial-zero-3d}
\end{figure}

As shown in \figref{fig:nontrivial-zero-3d},
there are exactly 4 kinds of local coefficient representation for the zero function in a single element.
The value at each vertex represents the coefficient for the corresponding 
node--based function in $\mfB$.
If any global coefficient representation for the zero function is restricted in an element,
then it has to be a linear combination of these 4 elementary representations
which are denoted by $\mcA, \mcX, \mcY$ and $\mcZ$, respectively.
In other words,
any global representation for the zero function is obtained
by a consecutive extension of local representation in an appropriate way.

For $D= \mcA, \mcX\cup\mcA, \mcY\cup\mcA,
\mcZ\cup\mcA,  \mcX\cup\mcY\cup\mcZ\cup\mcA,$ designate $\mcS_{\mathcal D}$ the following subspace consisting of global representation:
\begin{eqnarray*}
\mcS_{\mathcal D}:= \left\{ \mathbf{c} \in \BbbR^{|\mfB|} ~|~ 
  \left.\mathbf{c}\right|_{Q} \in \Span \{\mathcal D\}
  \forany Q\in\Tau_h \right\}.
\end{eqnarray*}
\begin{remark}
The definition of $B^{\mfB}_h$ implies
$\ker B^{\mfB}_h = \mcS_{\mcX\cup\mcY\cup\mcZ\cup\mcA}$.
\end{remark}
We then have the following matching conditions
on every face which is shared by two adjacent elements.
\begin{lemma}\label{lem:relation-s-xyza}
Let $\mathbf{c} \in \mcS_{\mcX\cup\mcY\cup\mcZ\cup\mcA}$, and
$c^{\mcX}_{ijk}, c^{\mcY}_{ijk}, c^{\mcZ}_{ijk}, c^{\mcA}_{ijk}$
denote coefficients of $\mathbf{c}$ in an element $Q_{ijk}\in\Tau_h$
for $\mcX$, $\mcY$, $\mcZ$, $\mcA$, respectively,
{\em i.e.},
$
\left.\mathbf{c}\right|_{Q_{ijk}} 
= 
	c^{\mcX}_{ijk} \mcX
	+
	c^{\mcY}_{ijk} \mcY
	+
	c^{\mcZ}_{ijk} \mcZ
	+
	c^{\mcA}_{ijk} \mcA.
$
Then for all $1\le i \le N_x$, $1\le j \le N_y$, $1\le k \le N_z$,
\begin{subeqnarray}\label{eq:relations} 
&&c^{\mcX}_{ijk} - c^{\mcA}_{ijk}
=
	c^{\mcX}_{(i+1)jk} + c^{\mcA}_{(i+1)jk}, \quad
c^{\mcY}_{ijk}
=
	- c^{\mcY}_{(i+1)jk}, \quad
c^{\mcZ}_{ijk}
=
	- c^{\mcZ}_{(i+1)jk}, \slabel{eq:relation-x-direction}\\
\slabel{eq:relation-y-direction}
&&c^{\mcY}_{ijk} - c^{\mcA}_{ijk}
=
	c^{\mcY}_{i(j+1)k} + c^{\mcA}_{i(j+1)k}, \quad
c^{\mcZ}_{ijk}
=
	- c^{\mcZ}_{i(j+1)k},\quad
c^{\mcX}_{ijk}
=
	- c^{\mcX}_{i(j+1)k},\\
\slabel{eq:relation-z-direction}
&&c^{\mcZ}_{ijk} - c^{\mcA}_{ijk}
=
	c^{\mcZ}_{ij(k+1)} + c^{\mcA}_{ij(k+1)};\quad
c^{\mcX}_{ijk}
=
  - c^{\mcX}_{ij(k+1)};\quad
c^{\mcY}_{ijk}
=
	- c^{\mcY}_{ij(k+1)}.
      \end{subeqnarray}
      Here all indices are understood up to modulo $N_x$, $N_y$, $N_z$, respectively,
due to periodicity.
\end{lemma}

\begin{remark}\label{rmk:well-definedness-s-xyza}
Conversely, 
local relations \eqref{eq:relation-x-direction}--\eqref{eq:relation-z-direction}
in \lemref{lem:relation-s-xyza} 
for all $1\le i \le N_x$, $1\le j \le N_y$, $1\le k \le N_z$ imply
the well-definedness of $\mathbf{c} \in \mcS_{\mcX\cup\mcY\cup\mcZ\cup\mcA}$,
{\em i.e.}, on each face shared by two adjacent elements the vertex values are matching.
\end{remark}

\begin{proof}[Proof of \lemref{lem:relation-s-xyza}]
Two elements $Q_{ijk}$ and $Q_{(i+1)jk}$ are adjacent in {the topological}  $x$-direction,
and sharing a common face {topologically} perpendicular to the $x$-axis.
Thus the vertex values on the right face of {the topologically} left element $Q_{ijk}$ have to be matched with
the vertex values on the {the topologically} left face of {the topologically} right element $Q_{(i+1)jk}$.
Since there are 4 nodes in the common face,
we have 4 equations in 8 variables:
\begin{subequations}\label{eq:relation-x-direction-node}
\begin{align}
- c^{\mcX}_{ijk} + c^{\mcY}_{ijk} + c^{\mcZ}_{ijk} + c^{\mcA}_{ijk}
&=
	- c^{\mcX}_{(i+1)jk} - c^{\mcY}_{(i+1)jk} - c^{\mcZ}_{(i+1)jk} - c^{\mcA}_{(i+1)jk}, \\
c^{\mcX}_{ijk} + c^{\mcY}_{ijk} - c^{\mcZ}_{ijk} - c^{\mcA}_{ijk}
&=
	\phantom{+} c^{\mcX}_{(i+1)jk} - c^{\mcY}_{(i+1)jk} + c^{\mcZ}_{(i+1)jk} + c^{\mcA}_{(i+1)jk}, \\
c^{\mcX}_{ijk} - c^{\mcY}_{ijk} + c^{\mcZ}_{ijk} - c^{\mcA}_{ijk}
&=
	\phantom{+} c^{\mcX}_{(i+1)jk} + c^{\mcY}_{(i+1)jk} - c^{\mcZ}_{(i+1)jk} + c^{\mcA}_{(i+1)jk}, \\
- c^{\mcX}_{ijk} - c^{\mcY}_{ijk} - c^{\mcZ}_{ijk} + c^{\mcA}_{ijk}
&=
	- c^{\mcX}_{(i+1)jk} + c^{\mcY}_{(i+1)jk} + c^{\mcZ}_{(i+1)jk} - c^{\mcA}_{(i+1)jk}.
\end{align}
\end{subequations}
Simple calculation shows that \eqref{eq:relation-x-direction-node} are equivalent to \eqref{eq:relation-x-direction}.
Similarly, considering faces {topologically} perpendicular to the $y$- and $z$-directions,
we get \eqref{eq:relation-y-direction} and \eqref{eq:relation-z-direction}, respectively.
\end{proof}

The next decomposition theorem is essential for the dimension analysis
in the 3D case.

\begin{theorem}[Decomposition Theorem]\label{thm:quotient-space-decomposition}
The quotient space $\mcS_{\mcX\cup\mcY\cup\mcZ\cup\mcA}/\mcS_{\mcA}$
can be decomposed as
\begin{align}
\mcS_{\mcX\cup\mcY\cup\mcZ\cup\mcA}/\mcS_{\mcA}
=
	\mcS_{\mcX\cup\mcA}/\mcS_{\mcA}
	\oplus
	\mcS_{\mcY\cup\mcA}/\mcS_{\mcA}
	\oplus
	\mcS_{\mcZ\cup\mcA}/\mcS_{\mcA}.
\end{align}
\end{theorem}
\begin{proof}
It is clear that
$
\mcS_{\mcA} 
\subset
\mcS_{\mcX\cup\mcA}, 
\mcS_{\mcY\cup\mcA},
\mcS_{\mcZ\cup\mcA}
\subset
\mcS_{\mcX\cup\mcY\cup\mcZ\cup\mcA}
$
and
$
\mcS_{\mcX\cup\mcA} \cap \mcS_{\mcY\cup\mcA}
=
	\mcS_{\mcY\cup\mcA} \cap \mcS_{\mcZ\cup\mcA}
=
	\mcS_{\mcZ\cup\mcA} \cap \mcS_{\mcX\cup\mcA}
=
	\mcS_{\mcA}
$.
Thus it is enough to show that for any $\mathbf{c} \in \mcS_{\mcX\cup\mcY\cup\mcZ\cup\mcA}$,
there exist 
$\bu \in \mcS_{\mcX\cup\mcA}$,
$\mathbf{v} \in \mcS_{\mcY\cup\mcA}$,
$\mathbf{w} \in \mcS_{\mcZ\cup\mcA}$
such that
$
\mathbf{c} \in \bu + \mathbf{v} + \mathbf{w} + \mcS_{\mcA}.
$

Let $c^{\mcX}_{ijk}$, $c^{\mcY}_{ijk}$, $c^{\mcZ}_{ijk}$, $c^{\mcA}_{ijk}$
denote the coefficients of $\mathbf{c}$ in $Q_{ijk}\in\Tau_h$
for $\mcX$, $\mcY$, $\mcZ$, $\mcA$, respectively,
{\em i.e.},
$
\left.\mathbf{c}\right|_{Q_{ijk}} 
= 
	c^{\mcX}_{ijk} \mcX
	+
	c^{\mathcal{nY}}_{ijk} \mcY
	+
	c^{\mcZ}_{ijk} \mcZ
	+
	c^{\mcA}_{ijk} \mcA.
$
Due to \lemref{lem:relation-s-xyza},
the relations \eqref{eq:relation-x-direction}--\eqref{eq:relation-z-direction} hold.
Now we construct $\bu$, $\mathbf{v}$, and $\mathbf{w}$.
First, define $\bu \in \BbbR^{|\mfB|}$ by
\begin{align}
\left.\bu\right|_{Q_{ijk}}
:=
	u^{\mcX}_{ijk} \mcX
	+
	u^{\mcA}_{ijk} \mcA
	\quad \text{where} \quad
	u^{\mcX}_{ijk} = c^{\mcX}_{ijk},
	u^{\mcA}_{ijk} = (-1)^{j+k} c^{\mcA}_{i11}.
\end{align}
We have $u^{\mcY}_{ijk} = u^{\mcZ}_{ijk} = 0$.
We can check the followings.
\begin{enumerate}
\item $\bu$ is well-defined, and belongs to $\mcS_{\mcX\cup\mcY\cup\mcZ\cup\mcA}$:
	See \rmkref{rmk:well-definedness-s-xyza}.
	For a face shared by two adjacent elements $Q_{ijk}$ and $Q_{(i+1)jk}$,
	\begin{align*}
	u^{\mcX}_{ijk} - u^{\mcA}_{ijk}
	&= c^{\mcX}_{ijk} - (-1)^{j+k}c^{\mcA}_{i11} \\
	&= c^{\mcX}_{(i+1)jk} + (-1)^{j+k}c^{\mcA}_{(i+1)11}
	= u^{\mcX}_{(i+1)jk} + u^{\mcA}_{(i+1)jk}.
	\end{align*}
	Thus $\bu$ is matching on all faces perpendicular to the $x$-axis.
	For the faces perpendicular to the $y$-axis, we have
	\begin{align*}
	& u^{\mcX}_{ijk}
	= c^{\mcX}_{ijk}
	= - c^{\mcX}_{i(j+1)k}
	= - u^{\mcX}_{i(j+1)k}, \\
	& - u^{\mcA}_{ijk}
	= - (-1)^{j+k} c^{\mcA}_{i11}
	= (-1)^{j+1+k} c^{\mcA}_{i11}
	= u^{\mcA}_{i(j+1)k},
	\end{align*}
	and similar for the faces perpendicular to the $z$-axis.
	Therefore $\bu$ is also matching along the $y$- and $z$-directions.
\item $\bu \in \mcS_{\mcX\cup\mcA}$:
	It is trivial due to the definition of $\bu$ and $\mcS_{\mcX\cup\mcA}$.
\end{enumerate}
Similarly to $\bu$, we define $\mathbf{v}$ and $\mathbf{w} \in \BbbR^{|\mfB|}$ by
\begin{align*}
& \left.\mathbf{v}\right|_{Q_{ijk}}
:=
	v^{\mcY}_{ijk} \mcY
	+
	v^{\mcA}_{ijk} \mcA
	\quad \text{where} \quad
	v^{\mcY}_{ijk} = c^{\mcY}_{ijk},
	v^{\mcA}_{ijk} =	(-1)^{i+k} c^{\mcA}_{1j1}, \\
& \left.\mathbf{w}\right|_{Q_{ijk}}
:=
	w^{\mcZ}_{ijk} \mcZ
	+
	w^{\mcA}_{ijk} \mcA
	\quad \text{where} \quad
	w^{\mcZ}_{ijk} = c^{\mcZ}_{ijk},
	w^{\mcA}_{ijk} = 	(-1)^{i+j} c^{\mcA}_{11k}.
\end{align*}
Then both $\mathbf{v}$ and $\mathbf{w}$ are well-defined, 
and $\mathbf{v} \in \mcS_{\mcY\cup\mcA}$,
$\mathbf{w} \in \mcS_{\mcZ\cup\mcA}$.
Thus
$\mathbf{c} - (\bu + \mathbf{v} + \mathbf{w}) \in \mcS_{\mcX\cup\mcY\cup\mcZ\cup\mcA}$.
We can conclude $\mathbf{c} - (\bu + \mathbf{v} + \mathbf{w}) \in \mcS_{\mcA}$
since for each $Q_{ijk}$,
\begin{align*}
\left.\mathbf{c} - (\bu + \mathbf{v} + \mathbf{w})\right|_{Q_{ijk}}
= \left(c^{\mcA}_{ijk} - (-1)^{j+k} c^{\mcA}_{i11} - (-1)^{i+k} c^{\mcA}_{1j1} - (-1)^{i+j} c^{\mcA}_{11k}\right) \mcA.
\end{align*}
\end{proof}

\begin{corollary}\label{cor:dim-equation}
$\dim \ker B^{\mfB}_h 
= \dim \mcS_{\mcX\cup\mcA} 
+ \dim \mcS_{\mcY\cup\mcA}
+ \dim \mcS_{\mcZ\cup\mcA}
- 2 \dim \mcS_{\mcA}$.
\end{corollary}

The following lemmas explain the dimension of subspaces
which depends on parity of the discretization numbers.

\begin{lemma}\label{lem:dim-s-xa}
(The dimension of $\mcS_{\mcX\cup\mcA}$, $\mcS_{\mcY\cup\mcA}$, $\mcS_{\mcZ\cup\mcA}$)
\begin{subeqnarray}
\dim \mcS_{\mcX\cup\mcA} &=& N_x \eps(N_y)\eps(N_z), \\
\dim \mcS_{\mcY\cup\mcA} &=&\eps(N_x) N_y\eps(N_z), \\
\dim \mcS_{\mcZ\cup\mcA} &=&\eps(N_x) \eps(N_y)N_z.
\end{subeqnarray}
\end{lemma}

\begin{proof}
It is enough to show the claim for $\mcS_{\mcX\cup\mcA}$,
since the others are similar.
Let $\mathbf{c} \in \mcS_{\mcX\cup\mcA}$
where
$\left.\mathbf{c}\right|_{Q_{ijk}} = c^{\mcX}_{ijk} \mcX + c^{\mcA}_{ijk} \mcA$
in each cube $Q_{ijk}\in\Tau_h$.
By applying the matching conditions \eqref{eq:relation-y-direction} and \eqref{eq:relation-z-direction} consecutively,
it can be shown
\begin{align*}
c^{\mcX}_{ijk} = (-1)^{j+k} c^{\mcX}_{i11}
\quad \text{and} \quad
c^{\mcA}_{ijk} = (-1)^{j+k} c^{\mcA}_{i11}.
\end{align*}
Consider $N_x+1$ combined surfaces such that
each of them consists of $N_y \times N_z$ faces in $\Tau_h$, and
is lying on the same hyperplane perpendicular to the $x$-axis.
The above relations imply that 
on each surface the coefficients for node--based functions
are all the same, but with alternating sign like a checkerboard pattern at nodes, not on faces.
Due to the identification between boundary nodes in the $y$- and $z$-directions,
all coefficients vanish unless both $N_y$ and $N_z$ are even.

Under the case of even $N_y$ and $N_z$,
we consider a basis checkerboard pattern at nodes on a combined surface consisting of $+1$ and $-1$, alternatively,
as \figref{fig:cube-combined-surface} (a) shows.
In the figure, the plus and minus sign at nodes represent
the positive value one, and the negative value one, respectively.
We get $N_x+1$ checkerboard patterns on $N_x+1$ combined surfaces in series (\figref{fig:cube-combined-surface} (b)).
Based on the basis checkerboard pattern described above,
we can represent all coefficients on each combined surface by a single factor in real number.
Due to the identification between boundary nodes in the $x$-direction,
factors for the first and the last combined surface must be same.
Then the series of $N_x+1$ checkerboard patterns compose
a global representation for a function in $\mcS_{\mcX\cup\mcA}$ (\figref{fig:cube-combined-surface} (c)).
Conversely, for the $N_x+1$ combined surfaces which are perpendicular
to the $x$-axis
and the basis checkerboard pattern at nodes on surfaces,
suppose $N_x+1$ factors are given, where the first and the last of them are same.
Then we can determine unique $c^{\mcX}_{ijk}$ and $c^{\mcA}_{ijk}$,
for all $Q_{ijk}\in\Tau_h$.
Therefore, only in the case when both $N_y$ and $N_z$ are even,
$\mcS_{\mcX\cup\mcA}$ is equivalent to
$\{ \mathbf{v} \in \BbbR^{N_x+1} ~|~ v_1 = v_{N_x+1} \}$
and $\dim \mcS_{\mcX\cup\mcA} = N_x$ consequently.
\end{proof}

\begin{figure}[!ht]
\centering
\epsfig{figure=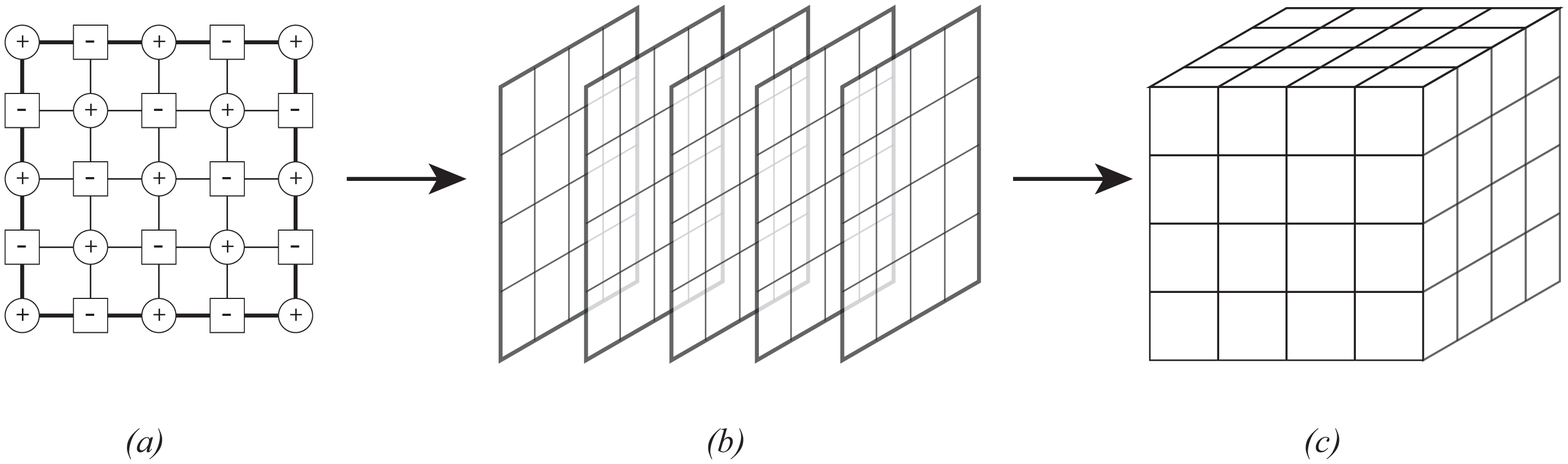,width=0.6\textwidth}
 \caption{Construction of a global representation for a function in $\mcS_{\mcX\cup\mcA}$}
 \label{fig:cube-combined-surface}
\end{figure}

\begin{lemma}\label{lem:dim-s-a}
(The dimension of $\mcS_{\mcA}$)
\begin{align}
\dim \mcS_{\mcA} = \eps(N_x)\eps(N_y)\eps(N_z).
\end{align}
\end{lemma}
\begin{proof}
Let $\mathbf{c} \in \mcS_{\mcA}$
where
$\left.\mathbf{c}\right|_{Q_{ijk}} = c^{\mcA}_{ijk} \mcA$
in each cube $Q_{ijk}$.
By applying the matching conditions \eqref{eq:relation-x-direction}--\eqref{eq:relation-z-direction} consecutively,
it is shown
$
c^{\mcA}_{ijk} = (-1)^{i+j+k+1} c^{\mcA}_{111}.
$
Due to the identification of boundary nodes in the $x$-, $y$-, and $z$-directions,
all coefficients vanish unless all $N_x$, $N_y$ and $N_z$ are even.
In the case of all even $N_x$, $N_y$ and $N_z$,
it is easily shown that the coefficients form
a multiple of {\it the 3D checkerboard pattern} at nodes.
Therefore $\dim \mcS_{\mcA} = 1$.
\end{proof}

The following proposition is
a direct consequence of \corref{cor:dim-equation}, \lemrefs{lem:dim-s-xa}{lem:dim-s-a}.
\begin{proposition}\label{prop:dim-ker-3d}
(The dimensions of $\ker B^{\mfB}_h$, $V^{\mfB,h}_{\#}$ in 3D)
\begin{eqnarray*}
\dim \ker B^{\mfB}_h
&=&N_x\eps(N_y)\eps(N_z)+\eps(N_x)N_y\eps(N_z)+\eps(N_x)\eps(N_y)N_z
  -2 \eps(N_x)\eps(N_y)\eps(N_z),\\
\dim V^{\mfB,h}_{\#} &=& N_x N_y N_z -
  \left[N_x\eps(N_y)\eps(N_z)+\eps(N_x)N_y\eps(N_z)+\eps(N_x)\eps(N_y)N_z\right.\\
&&\left. \qquad\qquad  -2 \eps(N_x)\eps(N_y)\eps(N_z)\right].
\end{eqnarray*}
\end{proposition}

\subsection{A basis for $V^h_{\#}$ in 3D}

\proprefs{prop:dim-periodic-3d}{prop:dim-ker-3d} imply that
$V^{\mfB,h}_{\#}$ is a proper subset of $V^h_{\#}$
if at most one of $N_x$, $N_y$, and $N_z$ is odd.
Furthermore,
if all $N_x$, $N_y$, and $N_z$ are even,
then there exist $2(N_x + N_y + N_z)-3$ complementary basis functions for $V^h_{\#}$,
not belonging to $V^{\mfB,h}_{\#}$.
If only $N_\iota$ is odd,
then the number of complementary basis functions for $V^h_{\#}$ is $2 N_\iota$.
In other cases,
$V^{\mfB,h}_{\#}$ is equal to $V^h_{\#}$.
We will discuss about the complementary basis functions below.

\begin{theorem}\label{thm:complementary-basis-3d}(A complementary basis for $V^h_{\#}$ in 3D)
Suppose a rectangular domain $\O$ is given with a triangulation $\Tau_h$,
consisting of cubes, in which the number of elements along coordinates are $N_x$, $N_y$, $N_z$.
For given $1 \le i \le N_x$,
let $\O^x_i$ be the subdomain consisting of $N_y \times N_z$ cubes
whose discrete position in $x$--coordinate are all same to $i$.
Let $(\psi^x_i)_y$ denote a piecewise linear function in $V^h_{\#}$,
whose support is $\O^x_i$, such that
it has nonzero barycenter values only on faces perpendicular to the $y$-axis,
and all the nonzero barycenter values are $1$ with alternating
sign in the $y$- and $z$-directions.
We can consider $(\psi^x_i)_z$, $(\psi^y_j)_x$, $(\psi^y_j)_z$, $(\psi^z_k)_x$, $(\psi^z_k)_y$
in similar manner.
The followings hold.
\begin{enumerate}
\item If all $N_x$, $N_y$, and $N_z$ are even,
then $V^{\mfB,h}_{\#}$ is a proper subset of $V^h_{\#}$.
The union of
\vspace{-0.25cm}
\begin{itemize}
\item any $N_x + N_y -1$ among $\mfA_z := \{(\psi^x_i)_z, (\psi^y_j)_z\}_{1\le i \le N_x, 1\le j \le N_y}$,
\item any $N_y + N_z -1$ among $\mfA_x := \{(\psi^y_j)_x, (\psi^z_k)_x\}_{1\le j \le N_y, 1\le k \le N_z}$, and
\item any $N_z + N_x -1$ among $\mfA_y := \{(\psi^z_k)_y, (\psi^x_i)_y\}_{1\le i \le N_x, 1\le k \le N_z}$
\end{itemize}
is a complementary basis for $V^h_{\#}$, not belonging to $V^{\mfB,h}_{\#}$.
\item If only $N_\iota$ is odd (and $N_\mu, N_\nu$ are even),
then $V^{\mfB,h}_{\#}$ is a proper subset of $V^h_{\#}$.
Moreover, $\{(\psi^\iota_j)_\mu, (\psi^\iota_j)_\nu\}_{1\le j \le N_\iota}$
is a complementary basis for $V^h_{\#}$, which is not contained in $V^{\mfB,h}_{\#}$.
\item Otherwise, $V^{\mfB,h}_{\#} = V^h_{\#}$.
\end{enumerate}
\end{theorem}

\begin{proof}
For the first case, suppose that all $N_x$, $N_y$, and $N_z$ are even.
Note that all nonzero barycenter values of $(\psi^x_i)_y$ are lying on the faces perpendicular to only one axis
with alternating sign,
as similar to the alternating function $\psi_x$ in the 2D case (\figref{fig:cube-alternating-function} (a), (b)),
and its support is $\O^x_i$ (\figref{fig:cube-alternating-function} (c)).
Using a similar argument as in the 2D case,
it is easily shown that
$(\psi^x_i)_y$ is well-defined, and not belonging to $V^{\mfB,h}_{\#}$
since $N_y$ and $N_z$ are even.
A similar property holds for $(\psi^x_i)_z$, 
a piecewise linear function in $V^h_{\#}$ whose support is $\O^x_i$
and which has nonzero barycenter values as $1$ only on faces
perpendicular to the $z$-axis
with alternating sign in the $y$- and $z$-directions.
Thus
there exist $2N_x$ alternating functions, $\{(\psi^x_i)_y, (\psi^x_i)_z\}_{1\le i\le N_x}$,
for $V^h_{\#}$ associated with strips perpendicular to the $x$-axis.
By considering other strips perpendicular to the $y$- or $z$-axis,
we can find out $2(N_x + N_y + N_z)$ alternating functions for $V^h_{\#}$,
not belonging to $V^{\mfB,h}_{\#}$:
$\{(\psi^x_i)_y, (\psi^x_i)_z, (\psi^y_j)_x, (\psi^y_j)_z, (\psi^z_k)_x, (\psi^z_k)_y\}_{1\le i \le N_x, 1\le j \le N_y, 1\le k \le N_z}$.

However,
there is a single relation between the alternating functions in each direction on subscript.
An alternating sum of $(\psi^x_i)_z$ in $1\le i \le N_x$ is equal to
that of $(\psi^y_j)_z$ in $1\le j \le N_y$.
And any $N_x + N_y -1$ among all $(\psi^x_i)_z$ and $(\psi^y_j)_z$
are linearly independent
due to their supports.
Similarly,
any $N_y + N_z -1$ among all $(\psi^y_j)_x$ and $(\psi^z_k)_x$ are linearly independent,
and so any $N_z + N_x -1$ among all $(\psi^z_k)_y$ and $(\psi^x_i)_y$ are.
Consequently,
suitably chosen~$2(N_x + N_y + N_z) -3$ alternating functions form a complementary basis for $V^h_{\#}$.

In the case of only one odd $N_\iota$ (and even $N_\mu$, $N_\nu$),
the set of all alternating functions associated to the strips
perpendicular to the $\iota$-axis,
$\{(\psi^\iota_j)_\mu, (\psi^\iota_j)_\nu \}_{1\le j \le N_\iota}$, are meaningful
because $N_\mu$ and $N_\nu$ are even.
\end{proof}

\begin{figure}[!ht]
\centering
\epsfig{figure=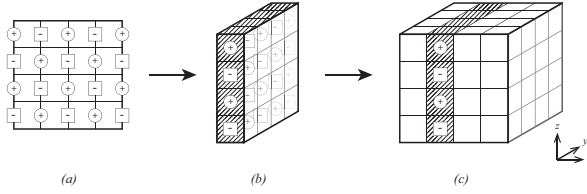,width=0.6\textwidth}
 \caption{Construction of an alternating function in 3D}
 \label{fig:cube-alternating-function}
\end{figure}

\subsection{Stiffness matrix associated with $\mfB$ in 3D}
The stiffness matrix $\bS^{\mfB}_h$ associated with $\mfB$ is defined as in \eqref{eq:def-stiffness-B} but in 3D space.
See \figref{fig:stencil-stiffness-matrix-3d} for the 3D local stencil for the stiffness matrix associated with $\mfB$.
\proprefs{prop:ker-decomposition}{prop:dim-ker-3d} lead the following proposition.

\begin{figure}[!ht]
\centering
\epsfig{figure=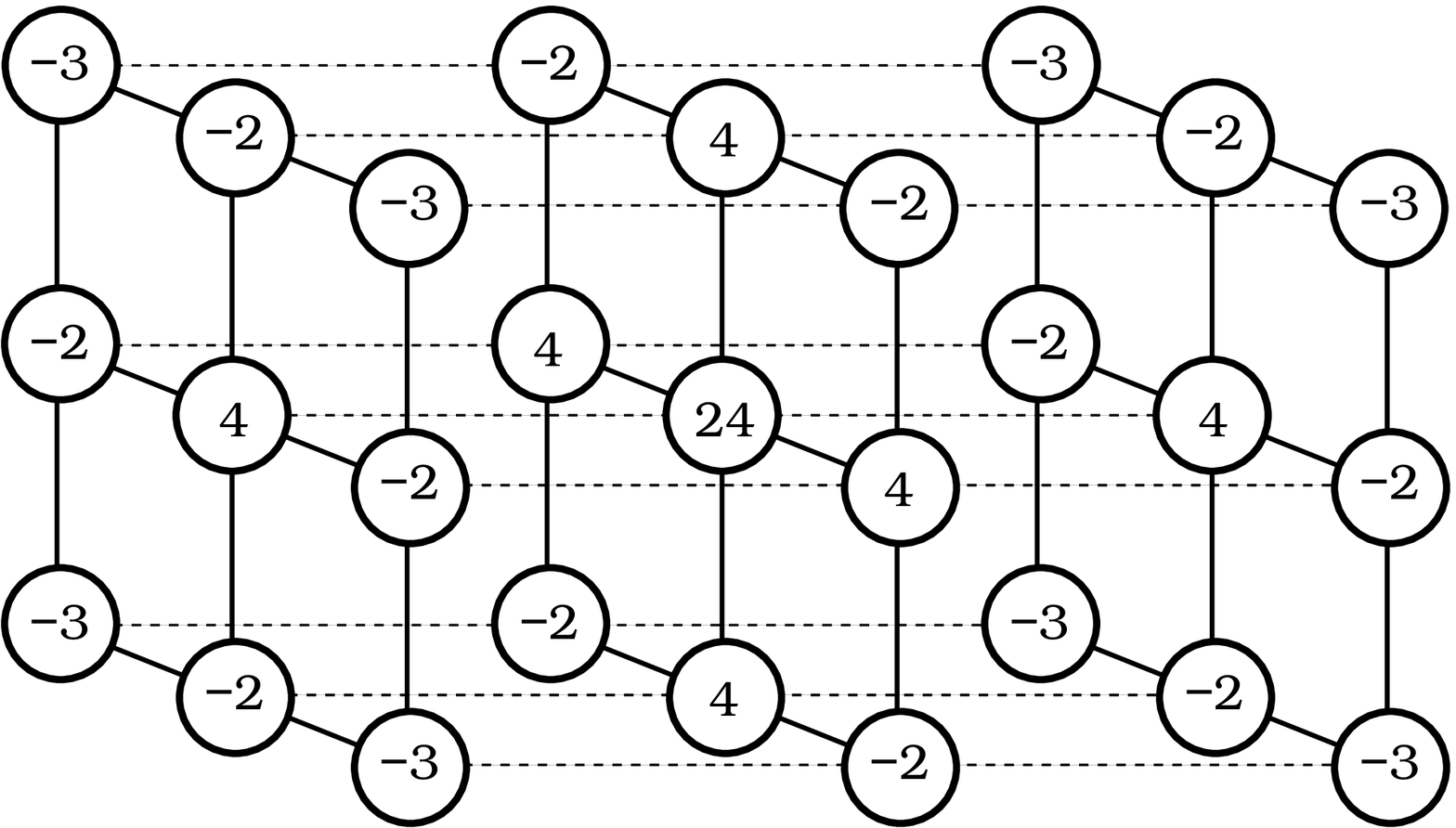,width=0.35\textwidth}
\caption{The stencil for $\bS^{\mfB}_h$ {with uniform cubes of
    size $h\times h\times h$ in 3D.}}
\label{fig:stencil-stiffness-matrix-3d}
\end{figure}

\begin{proposition}\label{prop:dim-ker-stiffness-3d}
(The dimension of $\ker \bS^{\mfB}_h$ in 3D)
\begin{eqnarray*}
\dim \ker \bS^{\mfB}_h
&= N_x\eps(N_y)\eps(N_z) + \eps(N_x)N_y\eps(N_z) +
                         \eps(N_x)\eps(N_y)N_z\\
&\qquad\qquad                         - 2 \eps(N_x)\eps(N_y)\eps(N_z) + 1.
\end{eqnarray*}
\end{proposition}

Let us assemble $\bS^{\mfB}_h$ for various combinations of $N_x$, $N_y$, and $N_z$.
The rank deficiency is computed by using MATLAB.
\tabref{tab:numerical-rank-deficiency-stiffness} shows
numerically obtained rank deficiency of $\bS^{\mfB}_h$ in 3D space.
Without loss of generality, it only represents combinations which hold $N_x \ge N_y \ge N_z$.
Numbers in red, blue, and black represent the case of all even discretizations and
the case of odd discretization in only one direction, and the other
cases, respectively.
These results imply that the rank deficiency pattern depend on parity
combination and confirm our theoretical result in \propref{prop:dim-ker-stiffness-3d}.

\begin{table}[!ht]
  \footnotesize
  \centering
  \begin{tabular}{ c | c | c c c c c c c }
  \hline \hline
  \multicolumn{2}{c|}{\multirow{2}{*}{$N_z = {\bf 2}$}} & $N_y$ \\ \cline{3-9}
  \multicolumn{2}{c|}{} & {\bf 2} & 3 & {\bf 4} & 5 & {\bf 6} & 7 & {\bf 8} \\ \hline
  $N_x$ & {\bf 2} & {\bf\color{red} 5} \\ \cline{2-9}
  & 3 & {\bf\color{blue} 4} & 1 \\ \cline{2-9}
  & {\bf 4} & {\bf\color{red} 7} & {\bf\color{blue} 4} & {\bf\color{red} 9} \\ \cline{2-9}
  & 5 & {\bf\color{blue} 6} & 1 & {\bf\color{blue} 6} & 1 \\ \cline{2-9}
  & {\bf 6} & {\bf\color{red} 9} & {\bf\color{blue} 4} & {\bf\color{red} 11} & {\bf\color{blue} 6} & {\bf\color{red} 13} \\ \cline{2-9}
  & 7 & {\bf\color{blue} 8} & 1 & {\bf\color{blue} 8} & 1 & {\bf\color{blue} 8} & 1 \\ \cline{2-9}
  & {\bf 8} & {\bf\color{red} 11} & {\bf\color{blue} 4} & {\bf\color{red} 13} & {\bf\color{blue} 6} & {\bf\color{red} 15} & {\bf\color{blue} 8} & {\bf\color{red} 17} \\ \hline \hline
  \multicolumn{2}{c|}{\multirow{2}{*}{$N_z = {\bf 4}$}} & \\ \cline{3-9}
  \multicolumn{2}{c|}{} & {\bf 2} & 3 & {\bf 4} & 5 & {\bf 6} & 7 & {\bf 8} \\ \hline
  & {\bf 2} &  \\ \cline{2-9}
  & 3 &  &  \\ \cline{2-9}
  & {\bf 4} &  &  & {\bf\color{red} 11} \\ \cline{2-9}
  & 5 &  &  & {\bf\color{blue} 6} & 1 \\ \cline{2-9}
  & {\bf 6} &  &  & {\bf\color{red} 13} & {\bf\color{blue} 6} & {\bf\color{red} 15} \\ \cline{2-9}
  & 7 &  &  & {\bf\color{blue} 8} & 1 & {\bf\color{blue} 8} & 1 \\ \cline{2-9}
  & {\bf 8} &  &  & {\bf\color{red} 15} & {\bf\color{blue} 6} & {\bf\color{red} 17} & {\bf\color{blue} 8} & {\bf\color{red} 19} \\ \hline \hline
  \end{tabular}
  \begin{tabular}{| c | c | c c c c c c c }
  \hline \hline
  \multicolumn{2}{|c|}{\multirow{2}{*}{$N_z = 3$}} & \phantom{$N_y$}\\ \cline{3-9}
  \multicolumn{2}{|c|}{} & {\bf 2} & 3 & {\bf 4} & 5 & {\bf 6} & 7 & {\bf 8} \\ \hline
  \phantom{$N_x$} & {\bf 2} \\ \cline{2-9}
  & 3 &  & 1 \\ \cline{2-9}
  & {\bf 4} &  & 1 & {\bf\color{blue} 4} \\ \cline{2-9}
  & 5 &  & 1 & 1 & 1 \\ \cline{2-9}
  & {\bf 6} &  & 1 & {\bf\color{blue} 4} & 1 & {\bf\color{blue} 4} \\ \cline{2-9}
  & 7 &  & 1 & 1 & 1 & 1 & 1 \\ \cline{2-9}
  & {\bf 8} &  & 1 & {\bf\color{blue} 4} & 1 & {\bf\color{blue} 4} & 1 & {\bf\color{blue} 4} \\ \hline \hline
  \multicolumn{2}{|c|}{\multirow{2}{*}{$N_z = 5$}} & \\ \cline{3-9}
  \multicolumn{2}{|c|}{} & {\bf 2} & 3 & {\bf 4} & 5 & {\bf 6} & 7 & {\bf 8} \\ \hline
  & {\bf 2} &  \\ \cline{2-9}
  & 3 &  &  \\ \cline{2-9}
  & {\bf 4} &  &  &  \\ \cline{2-9}
  & 5 &  &  &  & 1 \\ \cline{2-9}
  & {\bf 6} &  &  &  & 1 & {\bf\color{blue} 6} \\ \cline{2-9}
  & 7 &  &  &  & 1 & 1 & 1 \\ \cline{2-9}
  & {\bf 8} &  &  &  & 1 & {\bf\color{blue} 6} & 1 & {\bf\color{blue} 6} \\ \hline \hline
  \end{tabular}
  \caption{Numerically obtained rank deficiency of $\bS^{\mfB}_h$ in 3D}
  \label{tab:numerical-rank-deficiency-stiffness}
\end{table}

\subsection{Numerical schemes in 3D}
Consider again an elliptic problem with periodic BC \eqref{eq:weak-formulation-periodic}
with the compatibility condition $\int_\O f = 0$,
the corresponding weak formulation \eqref{eq:weak-formulation-periodic},
and the corresponding discrete weak formulation \eqref{eq:discrete-weak-formulation-periodic} in 3D.

Throughout this section,
we assume that all $N_x$, $N_y$, and $N_z$ are even.

\subsection*{Additional Notations \& Properties}
$\mfB^\flat$ again denotes a basis for {$V^{\mfB,h}_{\#}$}, a proper
subset of $\mfB$.
A constructive method for $\mfB^\flat$ will be given.
Let $\mfA$ and $\mfA^{\flat}$ be the set of all alternating functions,
and a complementary basis for $V^h_{\#}$ which consists of alternating functions
as in \thmref{thm:complementary-basis-3d}, respectively.
Without loss of generality,
we may write $\mfB^\flat = \{\phi_j \}_{j=1}^{|\mfB^\flat|}$,
$\mfB = \{\phi_j \}_{j=1}^{|\mfB|}$,
$\mfA^\flat = \{\psi_j\}_{j=1}^{|\mfA^\flat|}$, and
$\mfA = \{\psi_j\}_{j=1}^{|\mfA|}$.
Define two extended sets
$\mfE := \mfB \cup \mfA$, and
$\mfE^\flat := \mfB^\flat \cup \mfA^\flat$.
Even in the 3D case, $\mfE^\flat$ forms a basis for $V^h_{\#}$.
The characteristics of $\mfB^\flat$, $\mfB$, $\mfE^\flat$, and $\mfE$ in 3D
are summarized in \tabref{tab:summary-characteristic-3d}.

\begin{table}[!htb]
  \centering
  \footnotesize
  \begin{tabular}{ l | l | c | l }
    \hline \hline
    \multicolumn{1}{c|}{$\mcS$} & \multicolumn{1}{c|}{$|\mcS|$} & \multicolumn{1}{c|}{$\Span \mcS$} & \multicolumn{1}{c}{$\dim \Span \mcS$} \\ \hline \hline
    $\mfB^\flat$ & $N_x N_y N_z-(N_x+N_y+N_z)+2$ & \multirow{2}{*}{$V^{\mfB,h}_{\#}$} & \multirow{2}{*}{$N_x N_y N_z-(N_x+N_y+N_z)+2$} \\
    $\mfB$ & $N_x N_y N_z$ & & \\ \hline
    $\mfE^\flat$ & $N_x N_y N_z+(N_x+N_y+N_z)-1$ & \multirow{2}{*}{$V^h_{\#}$} & \multirow{2}{*}{$N_x N_y N_z+(N_x+N_y+N_z)-1$} \\
    $\mfE$ & $N_x N_y N_z+2(N_x+N_y+N_z)$ & & \\ \hline \hline
  \end{tabular}
  \caption{Characteristic of each test and trial function set $\mcS$ in 3D when all $N_x$, $N_y$, $N_z$ are even}
  \label{tab:summary-characteristic-3d}
\end{table}

\begin{remark}
Unlike in the 2D case,
$\mfA$ may not be linearly independent in the 3D case.
Thus
we use $\mfA^\flat$, a linearly independent subset, instead of $\mfA$ to construct $\mfE^\flat$ as a basis for $V^h_{\#}$.

\end{remark}

\begin{lemma}\label{lem:integral-property-B-A-3d}
Let $\mfB$ and $\mfA$ be as above. Then the followings hold.
\begin{enumerate}
\item $a_h (\phi, \psi) = 0 \forany \phi \in \mfB \forany \psi \in \mfA$.
\item $\int_\O \psi = 0 \forany \psi \in \mfA$.
\item There exists an $h$-independent constant $C$ such that
$\| \psi \|_0 \le C h^{1/2}$ and $| \psi |_{1,h} \le C h^{-1/2} \forany \psi \in \mfA$.
\end{enumerate}
\end{lemma}

\begin{remark}
The second equation in \lemref{lem:integral-property-B-A} does not
hold in the 3D case.
If $\mu = \nu$, then $a_h((\psi^\iota)_\mu, (\psi^\lambda)_\nu)$ does not vanish in general.
\end{remark}

For the 3D case, we define again $\bS^{\mfB^\flat}_h$, $\tilde{\bS}^{\mfB^\flat}_h$,
and $\bS^{\mfA}_h$
as in \eqref{eq:def-stiffness-B-flat}--\eqref{eq:def-stiffness-A}, respectively.
Furthermore we define $\bS^{\mfA^\flat}_h$,
the stiffness matrix associated with $\mfA^\flat$ similarly.
Define 
the linear systems $\mathcal{\tilde{L}}^{\mfE^\flat}_h$, $\mcL^{\mfE^\flat}_h$
as in \eqref{eq:scheme-1}, \eqref{eq:scheme-2},
with slight modification since $\mfE^\flat$ is equal to
$\mfB^\flat \cup \mfA^\flat$ in the 3D case.
Other linear systems
$\mcL^{\mfE}_h$, $\mcL^{\mfB}_h$ are defined as in
\eqref{eq:scheme-3}, \eqref{eq:scheme-4}.
The solutions
$\mathbf{\tilde{u}^\flat}$, $\mathbf{u^\flat}$, $\mathbf{u^\natural}$, $\mathbf{\bar{u}^\natural}$,
and the numerical solutions
$u_h$, $u^\flat_h$, $u^\natural_h$, $\bar{u}^\natural_h$
are defined as in
\eqref{eq:scheme-1}--\eqref{eq:scheme-2-zero mean value},
\eqref{eq:scheme-3}--\eqref{eq:scheme-4},
\eqref{eq:sol-scheme-1}, \eqref{eq:sol-scheme-2}, \eqref{eq:sol-scheme-3}, \eqref{eq:sol-scheme-4}.

The following describes relations between numerical solutions in 3D, as an analog of \thmref{thm:numerical-sols}.
\begin{theorem}\label{thm:numerical-sols-3d}
  {Let $(\Tau_h)_{0<h}$ be a family of uniform rectangular
  decomposition, that is, $\Tau_h =\widetilde\Tau_h$ for all $h.$
Assume that $N_x, N_y$ and $N_z$ are even.}
Let $u_h$, $u^\flat_h$, $u^\natural_h$, $\bar{u}^\natural_h$ be the numerical solutions of
\eqref{eq:weak-formulation-periodic} in 3D
as \eqref{eq:sol-scheme-1}, \eqref{eq:sol-scheme-2}, \eqref{eq:sol-scheme-3}, \eqref{eq:sol-scheme-4}, respectively,
with $\mfE^\flat = \mfB^\flat \cup \mfA^\flat$.
Then
$u_h = u^\flat_h = u^\natural_h$, and
\begin{equation*}
\|u^\natural_h - \bar{u}^\natural_h\|_0
\le C
	h^2 \|f\|_0, \
|u^\natural_h - \bar{u}^\natural_h|_{1,h}
\le C
	h \|f\|_0.
\end{equation*}
\end{theorem}

\begin{proof}
The equality between $u_h$ and $u^\flat_h$ can be proved as in the 2D case.
Since $\mfB^\flat$ is a basis for $V^{\mfB,h}_{\#}$,
there exist $t_{\ell j} \in \BbbR$ for $1\le \ell \le |\mfB|-|\mfB^\flat|$ and $1\le j \le |\mfB^\flat|$,
such that 
\begin{align}
\phi_{|\mfB^\flat|+\ell} = \sum_{j=1}^{|\mfB^\flat|} t_{\ell j} \phi_j.
\end{align}
Thus
$
\sum_{K\in\Tau_h} \nabla \phi_k \cdot \nabla \left(\phi_{|\mfB^\flat|+\ell} - \sum_{j=1}^{|\mfB^\flat|} t_{\ell j} \phi_j \right)  = 0
$
for all $k$, and it is simplified as
$
(\bS^\mfB_h)_{|\mfB^\flat|+\ell,k}
=
	\sum_{j=1}^{|\mfB^\flat|} t_{\ell j} (\bS^\mfB_h)_{jk}.
$
Let $\mathbf{T}$ denote a matrix of size $(|\mfB|-|\mfB^\flat|) \times |\mfB^\flat|$ 
such that $(\mathbf{T})_{\ell j} = t_{\ell j}$.
Then the last equation for $1\le \ell \le |\mfB|-|\mfB^\flat|$ and $1\le k \le |\mfB^\flat|$
can be expressed as a linear system
\begin{align}
[\bS^{\mfB}_h]_{|\mfB^\flat|+1:|\mfB|,1:|\mfB^\flat|}
&=
	\mathbf{T} [\bS^{\mfB}_h]_{1:|\mfB^\flat|, 1:|\mfB^\flat|}.
\end{align}
Note that $[\bS^{\mfB}_h]_{1:|\mfB^\flat|, 1:|\mfB^\flat|}$ is just equal to
$\bS^{\mfB^\flat}_h$.
Let $\begin{bmatrix} \mathbf{u^\flat}|_{\mfB^\flat} \\ \mathbf{0} \end{bmatrix}$ be
a trivial extension of $\mathbf{u^\flat}|_{\mfB^\flat}$ into a vector in $\BbbR^{|\mfB|}$
by padding zeros.
Then
\begin{align*}
\bS^{\mfB}_h \begin{bmatrix} \mathbf{u^\flat}|_{\mfB^\flat} \\ \mathbf{0} \end{bmatrix}
&=
	\begin{bmatrix}
	\bS^{\mfB^\flat}_h \mathbf{u^\flat}|_{\mfB^\flat} \\
	[\bS^{\mfB}_h]_{|\mfB^\flat|+1:|\mfB|,1:|\mfB^\flat|} \mathbf{u^\flat}|_{\mfB^\flat}
	\end{bmatrix}
=
	\begin{bmatrix}
	\bS^{\mfB^\flat}_h \mathbf{u^\flat}|_{\mfB^\flat} \\
	\mathbf{T} \bS^{\mfB^\flat}_h \mathbf{u^\flat}|_{\mfB^\flat}
	\end{bmatrix}
=
	\begin{bmatrix}
	\int_\O f \mfB^\flat \\
	\mathbf{T} \int_\O f \mfB^\flat
	\end{bmatrix}
\end{align*}
since $\bS^{\mfB^\flat}_h \mathbf{u^\flat}|_{\mfB^\flat} = \int_\O f \mfB^\flat$.
We can easily derive
\begin{align*}
\mathbf{T} \int_\O f \mfB^\flat
=
	\mathbf{T}
	\begin{bmatrix}
	\int_\O f \phi_1 \\
	\vdots \\
	\int_\O f \phi_{|\mfB^\flat|}
	\end{bmatrix}
=
	\begin{bmatrix}
	\int_\O f \sum_{j=1}^{|\mfB^\flat|} t_{1j} \phi_j \\
	\vdots \\
	\int_\O f \sum_{j=1}^{|\mfB^\flat|} t_{|\mfB^\flat| j} \phi_j
	\end{bmatrix}
=
	\begin{bmatrix}
	\int_\O f \phi_{|\mfB^\flat|+1} \\
	\vdots \\
	\int_\O f \phi_{|\mfB|}
	\end{bmatrix},
\end{align*}
which implies
$
\bS^{\mfB}_h \begin{bmatrix} \mathbf{u^\flat}|_{\mfB^\flat} \\ \mathbf{0} \end{bmatrix}
=
	\int_\O f \mfB.
$
In the same way we can obtain
$
\bS^{\mfA}_h \begin{bmatrix} \mathbf{u^\flat}|_{\mfA^\flat} \\ \mathbf{0} \end{bmatrix}
=
	\int_\O f \mfA,
$
and these equations derive $u^\natural_h = u^\flat_h$
by the same argument as in the 2D case.

For the last, consider the difference between $u^\natural_h$ and $\bar{u}^\natural_h$.
We can easily observe that
$u^\natural_h - \bar{u}^\natural_h = \left.\bu^\natural\right|_{\mfA} \mfA$,
and
$a_h(u^\natural_h - \bar{u}^\natural_h, \psi) = \int_\O f \psi$ for all $\psi \in \mfA$.
Thus
\begin{align*}
|u^\natural_h - \bar{u}^\natural_h|^2_{1,h}
&=
	a_h\left(u^\natural_h - \bar{u}^\natural_h, u^\natural_h - \bar{u}^\natural_h\right) \\
&=
	\int_\O f (u^\natural_h - \bar{u}^\natural_h)
\le C
	\|f\|_0 \, \|u^\natural_h - \bar{u}^\natural_h\|_0
= C
	h \, \|f\|_0 \, |u^\natural_h - \bar{u}^\natural_h|_{1,h} 
\end{align*}
due to the following lemma,
and we immediately obtain the difference in mesh-dependent norm, and in $L^2$-norm.
\end{proof}

\begin{lemma}
Let $\mcM^{\mfA}_h$ be the mass matrix associated with $\mfA$.
Then there exists an $h$-independent constant $C$ such that
$\mcM^{\mfA}_h = C h^2 \mcS^{\mfA}_h$.
In a consequence,
$\|v_h\|_{0} = C^{1/2} h |v_h|_{1,h}$ for all $v_h \in \Span \mfA$.
\end{lemma}
\begin{proof}
Remind that $(\psi^\iota_j)_\mu$ is the alternating function such that
the support is $\O^\iota_j$ and
the nonzero barycenter values are only lying on faces perpendicular to
the $\mu$-axis.
Thus only $\mu$-component of the piecewise gradient of $(\psi^\iota_j)_\mu$ survives.
It implies that $a_h((\psi^\iota_j)_\mu, (\psi^\lambda_k)_\nu) = 0$ if $\mu \not = \nu$.
Therefore we can consider $\mcS^{\mfA}_h$ as a block diagonal matrix:
$
\mcS^{\mfA}_h
=
	\begin{bmatrix}
	\mcS^{\mfA_x}_h & \mathbf{0} & \mathbf{0} \\
	\mathbf{0} & \mcS^{\mfA_y}_h & \mathbf{0} \\
	\mathbf{0} & \mathbf{0} & \mcS^{\mfA_z}_h
      \end{bmatrix},
$                                  
where $\mfA_x$, $\mfA_y$, $\mfA_z$ are defined as in \thmref{thm:complementary-basis-3d},
and $\mcS^{\mfA_x}_h$, $\mcS^{\mfA_y}_h$, $\mcS^{\mfA_z}_h$
are the stiffness matrices associated with the respective sets.

We can also consider $\mcM^{\mfA}_h$ as a block diagonal matrix
in the same structure,
since the following observation:
if $\mu \not = \lambda$, then
\begin{align*}
\left((\psi^\iota_j)_\mu, (\psi^\lambda_k)_\nu\right)_\O
&=
	\int_\O (\psi^\iota_j)_\mu \, (\psi^\lambda_k)_\nu 
=
	\sum_{Q \in \Tau_h(\O)} \int_Q (\psi^\iota_j)_\mu \, (\psi^\lambda_k)_\nu  \\
&=
	\sum_{Q \in \Tau_h(\O)}
	h
	\int_{Q_\mu} (\psi^\iota_j)_\mu \operatorname{d\mu}
	\int_{Q_\nu} (\psi^\lambda_k)_\nu \operatorname{d\nu} = 0.
\end{align*}
Set
$
\mcM^{\mfA}_h
=
	\begin{bmatrix}
	\mcM^{\mfA_x}_h & \mathbf{0} & \mathbf{0} \\
	\mathbf{0} & \mcM^{\mfA_y}_h & \mathbf{0} \\
	\mathbf{0} & \mathbf{0} & \mcM^{\mfA_z}_h
      \end{bmatrix},
      $
where $\mcM^{\mfA_x}_h$, $\mcM^{\mfA_y}_h$, $\mcM^{\mfA_z}_h$
are the mass matrices associated with the respective sets.
Therefore,
it is enough to show
$\mcM^{\mfA_\mu}_h = C h^2 \mcS^{\mfA_\mu}_h$
for each $\mu \in \{x,y,z\}$.

First, we consider the blocks associated with
$\mfA_x = \{(\psi^y_j)_x, (\psi^z_k)_x\}$ for $1\le j \le N_y,$ $1\le k \le N_z$.
The proof for other blocks is similar.
For any two alternating functions $(\psi^\iota_j)_x$ and
$(\psi^\lambda_k)_x$ in $\mfA_x$, we have

\noindent {\bf Case 1.} if $\iota = \lambda$ (let them be equal to $y$, without loss of generality), then
$$
a_h\left((\psi^y_j)_x, (\psi^y_k)_x\right)
=
	\sum_{Q \in \Tau_h(\O)} \int_Q \nabla (\psi^y_j)_x \cdot \nabla (\psi^y_k)_x 
=
	\sum_{Q \in \Tau_h(\O^y_j \cap \O^y_k)} \int_Q \frac4{h^2}
=
	4 N_x N_z h \delta_{jk},
$$
since the number of cubes in $\O^y_j$ is $N_x N_z$.
Here, $\delta_{jk}$ denotes the Kronecker delta.

\noindent {\bf Case 2.}
if $\iota \not = \lambda$ (let $\iota = y$ and $\lambda = z$, without loss of generality), then
\begin{align*}
a_h\left((\psi^y_j)_x, (\psi^z_k)_x\right)
=
	\sum_{Q \in \Tau_h(\O)} \int_Q \nabla (\psi^y_j)_x \cdot \nabla (\psi^z_k)_x 
=
	\sum_{Q \in \Tau_h(\O^y_j \cap \O^z_k)} \int_Q \frac4{h^2}
=
	4 N_x h,
\end{align*}
since the number of cubes in $\O^y_j \cap \O^z_k$ is $N_x$.
On the other hand, it is ready to see that
$$
\left((\psi^y_j)_x, (\psi^y_k)_x\right)_\O
=
	\sum_{Q \in \Tau_h(\O^y_j \cap \O^y_k)} \int_Q (\psi^y_j)_x \, (\psi^y_k)_x 
= \frac{N_x N_z h^3 \delta_{jk}}{3} ,$$  
and
$$
\left((\psi^y_j)_x, (\psi^z_k)_x\right)_\O
=
	\sum_{Q \in \Tau_h(\O^y_j \cap \O^z_k)} \int_Q (\psi^y_j)_x \, (\psi^z_k)_x 
=
	\frac{N_x h^3}{3} .
$$
Therefore
$\mcM^{\mfA_x}_h = \frac{1}{12} h^2 \mcS^{\mfA_x}_h$,
and the proof is completed.
\end{proof}

\subsection{Numerical results}
As mentioned before, our knowledge to construct a basis $\mfB^\flat$
for $V^{\mfB,h}_{\#}$ explicitly in 3D is lacking.
Thus we only use the scheme option 4 for our numerical test.
\begin{example}\label{ex:3}
Consider \eqref{eq:weak-formulation-periodic} on the domain $\O=(0,1)^3$ with
the exact solution $u(x,y,z) = \sin(2\pi x)
\sin(2\pi y) \sin(2\pi z)$.
\end{example}
The numerical results for Example \ref{ex:3} given in
\tabref{tab:elliptic-3d} confirm our theoretical results.

\begin{table}[htb!]
  \footnotesize
  \centering
  \begin{tabular}{ l | c  c | c  c }
    \hline \hline
    & \multicolumn{4}{c}{Opt 4} \\ \cline{2-5} 
    \multicolumn{1}{c|}{$h$} & $|u-\bar{u}^\natural_h|_{1,h}$ & order & $\|u-\bar{u}^\natural_h\|_0$ & order \\ \hline
    1/8 & 1.505E-00 & - & 3.848E-02 & - \\
    1/16 & 7.550E-01 & 0.995 & 9.716E-03 & 1.986 \\
    1/32 & 3.777E-01 & 0.999 & 2.434E-03 & 1.997 \\
    1/64 & 1.889E-01 & 1.000 & 6.089E-04 & 1.999 \\
    1/128 & 9.443E-02 & 1.000 & 1.523E-04 & 2.000 \\
    \hline \hline
  \end{tabular}
  \caption{{Numerical results for Example \ref{ex:3}.}}
  \label{tab:elliptic-3d}
\end{table}

\bigskip
\section*{Acknowledgments}
This research was supported in part by National Research Foundations
(NRF-2017R1A2B3012506 and NRF-2015M3C4A7065662).
\bigskip

\def\cprime{$'$}

\end{document}

\bibliographystyle{abbrv}
\bibliography{ms}

\begin{thebibliography}{10}

\bibitem{abdulle2012heterogeneous}
A.~Abdulle, E.~Weinan, B.~Engquist, and E.~Vanden-Eijnden.
\newblock The heterogeneous multiscale method.
\newblock {\em Acta Numerica}, 21:1--87, 2012.

\bibitem{altmann-carstensen-p1nc-tri-quad}
R.~Altmann and C.~Carstensen.
\newblock {$P_1$}-nonconforming finite elements on triangulations into
  triangles and quadrilaterals.
\newblock {\em SIAM J. Numer. Anal.}, 50(2):418--438, 2012.

\bibitem{axelsson-iterative}
O.~Axelsson.
\newblock {\em Iterative solution methods}.
\newblock Cambridge University Press, 1996.

\bibitem{babuska1976homogenization}
I.~Babu{\v{s}}ka.
\newblock Homogenization approach in engineering.
\newblock In {\em Computing methods in applied sciences and engineering}, pages
  137--153. Springer, 1976.

\bibitem{babuska1994special}
I.~Babu{\v{s}}ka, G.~Caloz, and J.~E. Osborn.
\newblock Special finite element methods for a class of second order elliptic
  problems with rough coefficients.
\newblock {\em SIAM J. Numer. Anal.}, 31(4):945--981, 1994.

\bibitem{bochev-lehoucq-pure-neumann}
P.~Bochev and R.~B. Lehoucq.
\newblock On the finite element solution of the pure {Neumann} problem.
\newblock {\em SIAM review}, 47(1):50--66, 2005.

\bibitem{campbell-meyer}
S.~L. Campbell and C.~D. Meyer.
\newblock {\em Generalized inverses of linear transformations}.
\newblock SIAM, 2009.

\bibitem{carstensen-hu-unifying-posteriori}
C.~Carstensen and J.~Hu.
\newblock A unifying theory of a posteriori error control for nonconforming
  finite element methods.
\newblock {\em Numer. Math.}, 107(3):473--502, 2007.

\bibitem{e2003heterogeneous}
W.~E and B.~Engquist.
\newblock The heterogeneous multiscale methods.
\newblock {\em Communications in Mathematical Sciences}, 1(1):87--132, 2003.

\bibitem{efendiev2013generalized}
Y.~Efendiev, J.~Galvis, and T.~Y. Hou.
\newblock Generalized multiscale finite element methods ({GMsFEM}).
\newblock {\em J. Comp. Phys.}, 251:116--135, 2013.

\bibitem{efendiev2014generalized}
Y.~Efendiev, J.~Galvis, G.~Li, and M.~Presho.
\newblock Generalized multiscale finite element methods: {Oversampling}
  strategies.
\newblock {\em International Journal for Multiscale Computational Engineering},
  12(6), 2014.

\bibitem{efendiev2009multiscale}
Y.~Efendiev and T.~Y. Hou.
\newblock {\em Multiscale finite element methods: theory and applications},
  volume~4.
\newblock Springer Science \& Business Media, 2009.

\bibitem{efendiev2004numerical}
Y.~Efendiev and A.~Pankov.
\newblock Numerical homogenization of nonlinear random parabolic operators.
\newblock {\em Multiscale Modeling \& Simulation}, 2(2):237--268, 2004.

\bibitem{engquist2008asymptotic}
B.~Engquist and P.~E. Souganidis.
\newblock Asymptotic and numerical homogenization.
\newblock {\em Acta Numerica}, 17:147--190, 2008.

\bibitem{feng-kim-nam-sheen-stabilized-stokes}
X.~Feng, I.~Kim, H.~Nam, and D.~Sheen.
\newblock Locally stabilized {$P_1$}-nonconforming quadrilateral and hexahedral
  finite element methods for the {Stokes} equations.
\newblock {\em J. Comput. Appl. Math.}, 236(5):714--727, 2011.

\bibitem{feng-li-he-liu-p1nc-fvm}
X.~Feng, R.~Li, Y.~He, and D.~Liu.
\newblock {$P_1$}-nonconforming quadrilateral finite volume methods for the
  semilinear elliptic equations.
\newblock {\em J. Sci. Comput.}, 52(3):519--545, 2012.

\bibitem{hou1997multiscale}
T.~Y. Hou and X.-H. Wu.
\newblock A multiscale finite element method for elliptic problems in composite
  materials and porous media.
\newblock {\em J. Comp. Phys.}, 134(1):169--189, 1997.

\bibitem{hughes1998variational}
T.~J.~R. Hughes, G.~R. Feij{\'o}o, L.~Mazzei, and J.-B. Quincy.
\newblock The variational multiscale method---a paradigm for computational
  mechanics.
\newblock {\em Comput. Methods Appl. Mech. Engrg.}, 166(1-2):3--24, 1998.

\bibitem{ipsen-meyer}
I.~C. Ipsen and C.~D. Meyer.
\newblock The idea behind {Krylov} methods.
\newblock {\em Amer. Math. Monthly}, pages 889--899, 1998.

\bibitem{jenny2003multiscale}
P.~Jenny, S.~Lee, and H.~Tchelepi.
\newblock Multi-scale finite-volume method for elliptic problems in subsurface
  flow simulation.
\newblock {\em J. Comp. Phys.}, 187(1):47--67, 2003.

\bibitem{ju-burkardt-mgmres}
L.~Ju and J.~Burkardt.
\newblock {MGMRES: Restarted GMRES solver for sparse linear systems}.
\newblock \url{http://people.sc.fsu.edu/~jburkardt/f_src/mgmres/mgmres.html}.
\newblock [Online; revision on 28-Aug-2012].

\bibitem{kaasschieter-pcg}
E.~F. Kaasschieter.
\newblock Preconditioned conjugate gradients for solving singular systems.
\newblock {\em J. Comput. Appl. Math.}, 24(1-2):265--275, 1988.

\bibitem{kim-yim-sheen}
S.~Kim, J.~Yim, and D.~Sheen.
\newblock Stable cheapest nonconforming finite elements for the {Stokes}
  equations.
\newblock {\em J. Comput. Appl. Math.}, 299:2--14, 2016.

\bibitem{lim-sheen-driven-cavity}
R.~Lim and D.~Sheen.
\newblock Nonconforming finite element method applied to the driven cavity
  problem.
\newblock {\em Comm. Comput. Phys.}, 21(4):1012--1038, 2017.

\bibitem{nam-choi-park-sheen-cheapest}
H.~Nam, H.~J. Choi, C.~Park, and D.~Sheen.
\newblock A cheapest nonconforming rectangular finite element for the
  stationary {Stokes} problem.
\newblock {\em Comput. Methods Appl. Mech. Engrg.}, 257:77--86, 2013.

\bibitem{park-locking}
C.~Park.
\newblock {\em A study on locking phenomena in finite element methods}.
\newblock PhD thesis, Department of Mathematics, Seoul National University,
  Seoul, Korea, 2002.

\bibitem{park-sheen-p1nc}
C.~Park and D.~Sheen.
\newblock {$P_1$}-nonconforming quadrilateral finite element methods for
  second-order elliptic problems.
\newblock {\em SIAM J. Numer. Anal.}, 41(2):624--640, 2003.

\bibitem{shi-pei-low-nonconforming-maxwell}
D.~Shi and L.~Pei.
\newblock Low order {Crouzeix}-{Raviart} type nonconforming finite element
  methods for approximating {Maxwell}'s equations.
\newblock {\em Int. J. Numer. Anal. Model}, 5(3):373--385, 2008.

\bibitem{yim-sheen-sim-nchmm}
J.~Yim, D.~Sheen, and I.~Sim.
\newblock {$P_1$}--nonconforming quadrilateral finite element space with
  periodic boundary conditions: {Part} {II.} {A}pplication to the nonconforming
  heterogeneous multiscale method.
\newblock {\em this jouranl}.
\newblock submitted.

\bibitem{zhang-lu-wei-uzawa-singular}
N.~Zhang, T.-T. Lu, and Y.~Wei.
\newblock Semi-convergence analysis of {Uzawa} methods for singular saddle
  point problems.
\newblock {\em J. Comput. Appl. Math.}, 255:334--345, 2014.

\end{thebibliography}

\end{document}